\documentclass[a4paper,11pt]{article}
\usepackage{adjustbox}
\usepackage{aligned-overset}
\usepackage{amsmath,amsthm}
\usepackage{authblk}
\usepackage[style=numeric-comp,maxbibnames=99,bibencoding=utf8,giveninits=true,backend=biber]{biblatex}
\usepackage{booktabs}
\usepackage{longtable}
\usepackage{multirow}
\usepackage{cancel}
\usepackage{cases}
\usepackage{colortbl}
\usepackage[hmargin={26mm,26mm},vmargin={30mm,35mm}]{geometry}
\usepackage{nicefrac}
\usepackage[colorlinks,allcolors={blue}]{hyperref}
\usepackage{paralist}
\usepackage{subcaption}
\usepackage{enumitem}
\usepackage[compact,small]{titlesec}
\usepackage{tikz-cd}
\usetikzlibrary{arrows}
\usepackage{multicol}
\usepackage{comment}

\usepackage{newtxtext}
\usepackage{newtxmath}

\bibliography{ddr_bgg}

\newcommand{\email}[1]{\href{mailto:#1}{#1}}

\numberwithin{equation}{section}

\newcommand{\adj}{\ast}

\newtheorem{theorem}{Theorem}
\newtheorem{proposition}[theorem]{Proposition}
\newtheorem{lemma}[theorem]{Lemma}

\theoremstyle{remark}
\newtheorem{remark}[theorem]{Remark}
\theoremstyle{definition}
\newtheorem{assumption}[theorem]{Assumption}

\newcommand{\Real}{\mathbb{R}}

\DeclareRobustCommand{\bvec}[1]{\boldsymbol{#1}}
\newcommand{\uvec}[1]{\underline{\bvec{#1}}}
\newcommand{\cvec}[1]{\bvec{\mathcal{#1}}}

\newcommand{\Cspace}[1]{\mathcal C_{#1}}

\DeclareMathOperator{\GRAD}{\bf grad}
\DeclareMathOperator{\CURL}{\bf curl}
\DeclareMathOperator{\DIV}{div}
\DeclareMathOperator{\VDIV}{\bf div}
\DeclareMathOperator{\ROT}{rot}
\DeclareMathOperator{\VROT}{\bf rot}
\DeclareMathOperator{\hess}{\nabla^2}

\newcommand{\compl}{{\rm c}}

\newcommand{\symbolproj}{\pi}
\newcommand{\lproj}[2]{\symbolproj_{\mathcal{P},#2}^{#1}}
\newcommand{\vlproj}[2]{\boldsymbol{\symbolproj}_{\cvec{P},#2}^{#1}}

\newcommand{\Rcproj}[2]{\bvec{\symbolproj}_{\cvec{R},#2}^{\compl,#1}}

\newcommand{\Xgrad}[2]{\underline{X}_{\GRAD,#2}^{#1}}
\newcommand{\Xcurl}[2]{\underline{\bvec{X}}_{\CURL,#2}^{#1}}

\newcommand{\Igrad}[2]{\underline{I}_{\GRAD,#2}^{#1}}

\newcommand{\uGh}[1]{\uvec{G}_h^{#1}}

\newcommand{\uCh}[1]{C_h^{#1}}

\newcommand{\edges}[1]{\mathcal{E}_{#1}}
\newcommand{\vertices}[1]{\mathcal{V}_{#1}}

\newcommand{\ET}{\edges{T}}

\newcommand{\VT}{\vertices{T}}

\newcommand{\VE}{\vertices{E}}

\newcommand{\normal}{\bvec{n}}
\newcommand{\tangent}{\bvec{t}}

\newcommand{\letterPoly}{\mathcal{P}}
\newcommand{\Poly}[1]{\letterPoly^{#1}}
\newcommand{\Roly}[1]{\cvec{R}^{#1}}
\newcommand{\cRoly}[1]{\cvec{R}^{\compl,#1}}

\newcommand{\stokes}{2}
\newcommand{\ddr}{1}

\newcommand{\norm}[2]{\|#2\|_{#1}}
\newcommand{\seminorm}[2]{|#2|_{#1}}
\newcommand{\vvvert}{\vert\kern-0.25ex\vert\kern-0.25ex\vert}
\newcommand{\tnorm}[2]{\vvvert #2\vvvert_{#1}}

\newcommand{\nSgrad}[2]{\tnorm{\stokes,#1}{#2}}
\newcommand{\nSrot}[2]{\tnorm{1,#1}{#2}}

\DeclareMathOperator{\Ker}{Ker}
\DeclareMathOperator{\Image}{Im}

\newcommand{\Th}{\mathcal{T}_h}
\newcommand{\Eh}{\mathcal{E}_h}
\newcommand{\Vh}{\mathcal{V}_h}

\newcommand{\Pgrad}[2]{\bvec{P}_{\GRAD,#2}^{#1}}
\newcommand{\Prot}[2]{\bvec{P}_{\VROT,#2}^{#1}}

\newcommand{\SPgradd}{P_{\stokes,T}^{k+3}}
\newcommand{\SPgrad}{P_{\stokes,T}^{k+1}}
\newcommand{\SPgradh}{P_{\stokes,h}^{k+1}}

\newcommand{\SProt}{\bvec{P}_{\ROT,T}^{k}}
\newcommand{\SProth}{\bvec{P}_{\ROT,h}^{k}}

\newcommand{\SC}[2]{\bvec{S}_{\CURL,#1}^{#2}}
\newcommand{\SR}[2]{\bvec{S}_{\ROT,#1}^{#2}}

\newcommand{\XSgrad}[1]{\underline{H}_{2}^{k}(#1)}
\newcommand{\XSgrado}[1]{\underline{H}_{2,0}^{k}(#1)}
\newcommand{\XSrot}[1]{\underline{H}_{1}^{k+1}(#1)^2}

\newcommand{\ISgrad}{\underline{I}_{\stokes,h}^{k}}
\newcommand{\ISgradT}{\underline{I}_{\stokes,T}^{k}}

\newcommand{\SGRAD}[1]{\underline{\bvec{G}}^{k}_{2,#1}}
\newcommand{\SROT}[1]{R^k_{1,#1}}

\newcommand{\Gqen}{G^{\normal}_{q,E}}
\newcommand{\Gqv}{\bvec G_{q,V}}
\newcommand{\Gqet}{G^{\tangent}_{2,E}}
\newcommand{\nablaT}{\bvec{G}_{2,T}^{k-1}}
\newcommand{\nablaTfull}{\bvec{G}_{2,T}^{k+2}}

\newcommand{\gammaSfull}[1]{\gamma^{k+3}_{\stokes,#1}}
\newcommand{\gammaS}{\gamma^{k+1}_{\stokes,\partial T}}

\newcommand{\Id}{\mathrm{Id}}
\newcommand{\sskw}{\mathop{\mathrm{sskw}}}
\newcommand{\sskwh}{\sskw\nolimits_h}

\newcommand{\XdRgrad}[1]{\underline{H}_{1}^{k+1}(#1)}
\newcommand{\XdRgrado}[1]{\underline{H}_{1,0}^{k+1}(#1)}
\newcommand{\XdRrot}[1]{\underline{\boldsymbol{H}}_{\operatorname{\boldsymbol{\mathrm{rot}}}}^{k+1}(#1)}
\newcommand{\XdRrotk}[1]{\underline{\boldsymbol{H}}_{\operatorname{\boldsymbol{\mathrm{rot}}}}^{k}(#1)}
\newcommand{\tXdRgrad}[1]{\XdRgrad{#1}^2}
\newcommand{\tXdRgrado}[1]{\XdRgrado{#1}^2}
\newcommand{\tXdRrot}[1]{\XdRrot{#1}^2}

\newcommand{\tGRAD}[1]{\underline{\bvec G}_{1,#1}^{k+1}}
\newcommand{\tROT}[1]{\bvec{R}_{\operatorname{\boldsymbol{\mathrm{rot}}},#1}^{k+1}}
\newcommand{\IdRgrad}{\underline{\bvec{I}}_{\ddr,h}^{k+1}}
\newcommand{\IdRgradT}{\underline{\bvec{I}}_{\ddr,T}^{k+1}}
\newcommand{\IdRrot}{\underline{\bvec{I}}_{\operatorname{\boldsymbol{\mathrm{rot}}},h}^{k+1}}
\newcommand{\IdRrotT}{\underline{\bvec{I}}_{\operatorname{\boldsymbol{\mathrm{rot}}},T}^{k+1}}

\newcommand{\gammadR}[1]{\gamma^{k+2}_{\ddr,#1}}

\newcommand{\XdRrotS}[1]{\underline{\boldsymbol{H}}_{\operatorname{\boldsymbol{\mathrm{rot}}}}^{k+1}(#1,\mathbb{S})}
\newcommand{\XdRrotSo}[1]{\underline{\boldsymbol{H}}_{\operatorname{\boldsymbol{\mathrm{rot}}},0}^{k+1}(#1,\mathbb{S})}

\newcommand{\Hess}[1]{\underline{\bvec{\mathcal{H}}}^{k+1}_{#1}}

\newcommand{\redex}[1]{\mathfrak{#1}}

\newcommand{\Egradh}{\underline{\redex{E}}_{\GRAD,h}}

\newcommand{\Rgradh}{\underline{\redex{R}}_{\GRAD,h}}

\newcommand{\Eroth}{\underline{\bvec{\redex{E}}}_{\ROT,h}}

\newcommand{\Rroth}{\underline{\redex{R}}_{\ROT,h}}

\newcommand{\dof}[2]{\underline{#1}_{#2}}
\newcommand{\dofh}[1]{\underline{#1}_h}
\newcommand{\dofT}[1]{\underline{#1}_T}

\newcommand{\injrot}{\bvec i_{\ROT,T}}

\newcommand{\acSgrad}{\mathcal{E}_{\stokes,h}}
\newcommand{\acSrot}{\mathcal{E}_{\ROT,h}}

\usepackage{marginnote}

\usepackage{pgfplots,pgfplotstable}
\usepackage{tikz-cd}

\graphicspath{{figures/}}

\newcommand{\drawPolygon}[7]{%
  \def\nsides{#1}%
  \def\radius{#2}%
  \def\showNodal{#3}%
  \def\showGradients{#4}%
  \def\pEdge{#5}%
  \def\varsEdge{1}%
  \def\pNormal{#6}%
  \def\varsNormal{1}%
  \def\nInterior{#7}%

  \foreach \i in {1,...,\nsides} {
    \pgfmathsetmacro{\angle}{360/\nsides*(\i - 1)}
    \pgfmathsetmacro{\x}{\radius*cos(\angle)}
    \pgfmathsetmacro{\y}{\radius*sin(\angle)}
    \coordinate (P\i) at (\x,\y);
  }

  \draw[thick] (P1)
  \foreach \i in {2,...,\nsides} { -- (P\i) }
  -- cycle;

  \ifnum\showNodal=1
  \foreach \i in {1,...,\nsides} {
    \fill (P\i) circle (2pt);
  }
  \fi

  \ifnum\showGradients=1
  \foreach \i in {1,...,\nsides} {
    \draw (P\i) circle (4pt);
  }
  \fi

  \ifnum\pEdge>-1
  \pgfmathtruncatemacro{\nEdge}{int((\pEdge+\varsEdge)!/(\pEdge!*\varsEdge!))}
  \foreach \i in {1,...,\nsides}{
    \pgfmathtruncatemacro{\j}{mod(\i,\nsides)+1}%
    \foreach \k in {1,...,\nEdge}{
      \pgfmathsetmacro{\t}{\k/(\nEdge+1)}
      \path (P\i) -- (P\j) coordinate[pos=\t] (E\i\k);
      \fill (E\i\k) circle (1.5pt);
    }
  }
  \fi

  \ifnum\pNormal>-1
  \pgfmathtruncatemacro{\nNorm}{int((\pNormal+\varsNormal)!/(\pNormal!*\varsNormal!))}
  \foreach \i in {1,...,\nsides}{
    \pgfmathtruncatemacro{\j}{mod(\i,\nsides)+1}
    \pgfmathsetmacro{\angleA}{360/\nsides*(\i - 1)}
    \pgfmathsetmacro{\angleB}{360/\nsides*(\j - 1)}
    \pgfmathsetmacro{\edgeang}{atan2(sin(\angleB)-sin(\angleA),cos(\angleB)-cos(\angleA))}
    \pgfmathsetmacro{\normang}{\edgeang-90}
    \foreach \k in {1,...,\nNorm}{
      \pgfmathsetmacro{\t}{\k/(\nNorm+1)}
      \path (P\i) -- (P\j) coordinate[pos=\t] (N\i\k);
      \draw[-{Latex[length=4pt]}, thick]
      (N\i\k) -- ++({0.6*cos(\normang)},{0.6*sin(\normang)});
    }
  }
  \fi

  \ifnum\nInterior>0
  \ifnum\nInterior=1
  \fill (0,0) circle (1.5pt);
  \else
  \pgfmathsetmacro{\rInt}{0.35*\radius}
  \foreach \k in {1,...,\nInterior} {
    \pgfmathsetmacro{\theta}{360/\nInterior*(\k - 1)}
    \pgfmathsetmacro{\xi}{\rInt*cos(\theta)}
    \pgfmathsetmacro{\yi}{\rInt*sin(\theta)}
    \fill (\xi,\yi) circle (1.5pt);
  }
  \fi
  \fi
}

\newcommand{\logLogSlopeTriangle}[5]{
  \pgfplotsextra{
    \pgfkeysgetvalue{/pgfplots/xmin}{\xmin}
    \pgfkeysgetvalue{/pgfplots/xmax}{\xmax}
    \pgfkeysgetvalue{/pgfplots/ymin}{\ymin}
    \pgfkeysgetvalue{/pgfplots/ymax}{\ymax}

    \pgfmathsetmacro{\xArel}{#1}
    \pgfmathsetmacro{\yArel}{#3}
    \pgfmathsetmacro{\xBrel}{#1-#2}
    \pgfmathsetmacro{\yBrel}{\yArel}
    \pgfmathsetmacro{\xCrel}{\xArel}

    \pgfmathsetmacro{\lnxB}{\xmin*(1-(#1-#2))+\xmax*(#1-#2)} 
    \pgfmathsetmacro{\lnxA}{\xmin*(1-#1)+\xmax*#1} 
    \pgfmathsetmacro{\lnyA}{\ymin*(1-#3)+\ymax*#3} 
    \pgfmathsetmacro{\lnyC}{\lnyA+#4*(\lnxA-\lnxB)}
    \pgfmathsetmacro{\yCrel}{\lnyC-\ymin)/(\ymax-\ymin)}

    \coordinate (A) at (rel axis cs:\xArel,\yArel);
    \coordinate (B) at (rel axis cs:\xBrel,\yBrel);
    \coordinate (C) at (rel axis cs:\xCrel,\yCrel);

    \draw[#5]   (A)-- node[pos=0.5,anchor=north] {\scriptsize{1}}
    (B)--
    (C)-- node[pos=0.,anchor=west] {\scriptsize{#4}} 
    cycle;
  }
}

\begin{document}

\title{Analytical properties of polygonal Stokes and BGG Hessian complexes with application to Kirchhoff--Love plates}

\author[1]{Daniele A. Di Pietro}
\author[1,2]{J\'er\^ome Droniou}
\author[3]{Kaibo Hu}
\author[1]{Arax Leroy}
\affil[1]{IMAG, Univ. Montpellier, CNRS, Montpellier, France, \email{daniele.di-pietro@umontpellier.fr}, \email{jerome.droniou@cnrs.fr}, \email{arax.leroy@umontpellier.fr}}
\affil[2]{School of Mathematics, Monash University, Melbourne, Australia}
\affil[3]{Mathematical Institute, University of Oxford, UK. \email{kaibo.hu@maths.ox.ac.uk}}
\maketitle

\begin{abstract}
  We develop and analyse a new arbitrary-order method for the Kirchhoff--Love plate problem on general polygonal meshes. The construction relies on a discrete Hessian complex obtained through the Bernstein--Gelfand--Gelfand construction, which stacks a Stokes complex on top of a tensorised de Rham complex.
  A key ingredient of the analysis is a detailed study of the former.
  We specifically establish primal and adjoint consistency estimates, as well as uniform Poincaré inequalities, and use these results to derive the analytical properties of the resulting discrete Hessian complex. The proposed method is then shown to be coercive and to converge with order $k+1$ with respect to the mesh size, where $k\ge 0$ is the polynomial degree of the complex. Numerical experiments on polygonal meshes confirm the theoretical convergence rates.
\end{abstract}


\section{Introduction}\label{sec:Introduction}

In this work, we develop a new arbitrary-order method on general polygonal meshes for the numerical approximation of the Kirchhoff--Love plate model.
The design of the discrete spaces and operators hinges on a discrete Hessian complex derived, through the Bernstein--Gelfand--Gelfand (BGG) construction \cite{Cap.Slovak.ea:01,Bernstein.Gelfand.ea:75}, from discrete Stokes and de Rham complexes \cite{Di-Pietro.Droniou.ea:26}.
After establishing preliminary results for the discrete Stokes and Hessian complexes in a general framework, we apply them to the analysis of the new method, whose performance is showcased through a comprehensive set of numerical examples.
\smallskip

Let $\Omega$ be a two-dimensional polygonal domain, and denote by $H_\ell(\Omega)$ the usual Sobolev space spanned by scalar functions $\Omega\to\Real$ whose derivatives up to order $\ell$ are square-integrable.
We consider the following BGG diagram, which stacks the \emph{Stokes complex} (i.e., a version of the de Rham complex with enhanced regularity) above a tensorised version of the standard de Rham complex:
\begin{equation}\label{eq:double.complex.cont}
  \begin{tikzcd}[column sep=2.5em]
    \text{Stokes:}
    & 0 \arrow{r}{}
    & H_2(\Omega)\arrow{r}{\GRAD}
    & H_1(\Omega)^2\arrow[dashed]{r}{\ROT}
    & L^2(\Omega)\arrow[dashed]{r}{}
    & 0 \\
    \text{de Rham:}
    & 0 \arrow[dashed]{r}{}
    & H_1(\Omega)^2\arrow{r}{\GRAD}\arrow[leftrightarrow]{ur}{\Id}
    & \boldsymbol{H}_{\VROT}(\Omega)^2\arrow{r}{\VROT}\arrow[dashed]{ur}{\sskw}
    & L^2(\Omega)^2\arrow{r}{}
    & 0
  \end{tikzcd}
\end{equation}
where, for suitable scalar-, vector- or tensor-valued functions (respectively denoted by $q$, $\bvec{v}$, $\bvec{\tau}$ and with usually indexed components),
\begin{align*}
  \GRAD q &\coloneq \begin{pmatrix} \partial_1 q \\ \partial_2 q
  \end{pmatrix},\quad
  \ROT \bvec{v} \coloneq \partial_1 v_2 - \partial_2 v_1,\quad
  \GRAD \bvec{v} \coloneq \begin{pmatrix}
    \partial_1 {v}_1 & \partial_2 {v}_1 \\
    \partial_1 {v}_2 & \partial_2 {v}_2
  \end{pmatrix},\\
  \quad
  \VROT \bvec{\tau} &\coloneq \begin{pmatrix}
    \ROT (\tau_{1,1}, \tau_{1,2})^\perp \\
    \ROT (\tau_{2,1}, \tau_{2,2})^\perp
  \end{pmatrix}=\begin{pmatrix}
  \partial_1\tau_{1,2}-\partial_2\tau_{1,1} \\
  \partial_1\tau_{2,2}-\partial_2\tau_{2,1}
  \end{pmatrix},\quad
  \sskw \bvec{\tau}=\tau_{12}-\tau_{21},
\end{align*}
and $ \boldsymbol{H}_{\VROT}(\Omega)^2\coloneq \{\bvec{\tau}\in L^2(\Omega)^{2\times 2}\,:\,\VROT\bvec{\tau}\in L^2(\Omega)^2\}$.

From this diagram, one can derive two complexes linked to plate models:
the \emph{twisted complex}, whose Hodge Laplacian yields the energy functional of the Reissner--Mindlin model for thick plates
and the \emph{BGG complex}, obtained by following the solid arrows from the top left to the bottom right in \eqref{eq:double.complex.cont}.
In the present setting, this corresponds to the Hessian complex:
\begin{equation}\label{BGG-Hessian-HD}
  \begin{tikzcd}
    \text{Hessian:}
    & 0 \arrow{r}
    & H_{2}(\Omega) \arrow{r}{\hess}
    & \boldsymbol H_{\VROT}(\Omega, \mathbb{S})\arrow{r}{\VROT}
    & L^{2}(\Omega)^{2}  \arrow{r}{} &0,
  \end{tikzcd}
\end{equation}
where $\boldsymbol H_{\VROT}(\Omega, \mathbb{S})$ denotes the space of square-integrable $\bvec{\tau}:\Omega\to \mathbb{S}$, where $\mathbb{S}$ is the space of $2\times 2$ symmetric matrices, such that $\VROT\bvec{\tau}$ is square-integrable.
The Hessian complex is linked to the Kirchhoff--Love model of thin plates.
For simplicity, we consider an isotropic material and normalise the model so that the material constants are omitted from the notation.
Given a surface load $f : \Omega \to \Real$ orthogonal to the plate, the unknown of the model is the deflection $u : \Omega \to \Real$ satisfying
\begin{equation}\label{eq:KL}
  \DIV \VDIV \hess u = f  \qquad \text{in $\Omega$,}
\end{equation}
where $\VDIV$ is the row-wise divergence operator acting on matrix-valued fields.
As shown in Section \ref{sec:hessian.complex}, the Hodge Laplacian associated with the complex \eqref{BGG-Hessian-HD} corresponds to complementing \eqref{eq:KL} with the following natural boundary conditions (see \cite{Pauly.Zulehner:20,Pauly.Schomburg:24} and also \cite[Section 1.3]{Pauly.Valli:25}), in which $\normal_\Omega$ is the unit outer normal to $\Omega$ on $\partial\Omega$:
\begin{equation}\label{eq:KL:bc:natural}
  (\nabla^2 u) \, \normal_{\Omega} = \bvec{0}\text{ and } \VDIV (\nabla^2 u) \cdot \normal_{\Omega} = 0 \quad \text{on $\partial \Omega$}.
\end{equation}
These conditions express the fact that the plate is unconstrained at the boundary.
Already at the continuous level, the study of the corresponding variational formulation involves technicalities; see, e.g., \cite{Pauly.Valli:25}.
For this reason, and since practical applications often involve plates that are clamped along the boundary, we instead focus on the essential boundary conditions
\begin{equation}\label{eq:KL:bc:essential}
  u = \nabla u \cdot \normal_{\Omega} = 0 \quad \text{on $\partial \Omega$}.
\end{equation}
This requires considering a variant of \eqref{BGG-Hessian-HD} in which each space is replaced with its zero-trace subspace.
\smallskip

The first main contribution of this work is the design, analysis, and numerical validation of a new scheme for problem \eqref{eq:KL} supplemented with the essential boundary conditions \eqref{eq:KL:bc:essential}.
The scheme supports general meshes composed of polygonal elements and relies on the discrete counterpart of the BGG diagram \eqref{eq:double.complex.cont} recently developed in \cite{Di-Pietro.Droniou.ea:26}. This construction yields a discrete analogue of the Hessian complex \eqref{BGG-Hessian-HD} whose cohomology is isomorphic to that of the continuous complex.
The method is coercive by design, and we prove that, when using the discrete Hessian complex of polynomial degree $k \ge 0$, the scheme converges with order $h^{k+1}$ (where $h$ denotes the mesh size) in the natural energy norm.

The finite element discretisation of plate models has a long history. Some of the earliest finite elements, such as the Argyris element \cite{argyris1968tuba} and the Hsieh--Clough--Tocher (HCT) element \cite{clough1965finite}, were developed for solving the biharmonic equations arising from the Kirchhoff--Love plate model. These elements are $C^1$-conforming and achieve the required continuity either through supersmoothness, i.e., higher-order derivative degrees of freedom at vertices (Argyris), or through a subdivision of each triangle into subelements (HCT).

Along a different line of development, the construction of stress--displacement finite element pairs for linear elasticity with polynomial shape functions remained a major challenge until the first successful construction by Arnold and Winther \cite{Arnold.Winther:02}. Their work was motivated by the elasticity complex, which in two dimensions can be viewed as a rotated analogue of the Hessian complex. The finite element complexes associated with both the Arnold--Winther and Hu--Zhang elements begin with the Argyris element \cite{Arnold.Falk.ea:06,Christiansen.Hu.ea:18}. Furthermore, a Hessian/elasticity complex can also be constructed starting from the HCT element \cite{Christiansen.Hu:23}, leading to a variant of the Johnson--Mercier element \cite{Johnson.Mercier:78}. Corresponding BGG diagrams were developed in \cite{Arnold.Falk.ea:06,Christiansen.Hu:23}. These diagrams additionally yield a discretisation of the Reissner--Mindlin plate model, which can be interpreted as the Hodge Laplacian of the twisted complex associated with the diagram \cite{Cap.Hu:24}.

Polygonal discretisation methods have also been developed for problem \eqref{eq:KL}--\eqref{eq:KL:bc:essential}.
In \cite{Brezzi.Marini:13} (see also \cite{Chinosi.Marini:16}), the authors propose and analyse a Virtual Element method based on a $C^1$-conforming space, which bears similarities with the approach considered here (see Remark \ref{rem:vem.comparison} below).
Despite using a Hessian reconstruction of one degree lower than the one considered here, the resulting method achieves comparable convergence rates.
Virtual Element methods based on non-$C^1$-conforming spaces are studied in \cite{Antonietti.Manzini.ea:18} and \cite{Zhao.Chen.ea:16}.
The fully discrete paradigm has also been explored to derive numerical schemes for the Kirchhoff--Love problem. We refer in particular to \cite{Bonaldi.Di-Pietro.ea:18}, where a Hybrid High-Order method is proposed, and to \cite{Di-Pietro.Droniou:23*2}, where, following the Discrete de Rham paradigm, the authors develop a discretisation for the mixed formulation of the problem.
We also mention \cite{wangwang:15}, which designs and analyses Weak Galerkin methods on polygonal meshes for the biharmonic equation.
\smallskip

The theoretical analysis of the proposed scheme heavily relies on the analytical properties of the discrete Stokes and Hessian complexes underlying its construction.
The second main contribution of this work is precisely a comprehensive study of these properties, completing the derivation and analysis of algebraic properties initiated in \cite{Di-Pietro.Droniou.ea:26}.
In particular, we establish primal and adjoint consistency estimates in the spirit of \cite[Section 6]{Di-Pietro.Droniou:23}, as well as uniform Poincar\'e-type inequalities.
\smallskip

The rest of this work is organised as follows.
In Section~\ref{sec:hessian.complex} we recall the definition of the Hodge Laplacian operator together with the associated Hessian complex problem~\eqref{BGG-Hessian-HD}. In Section~\ref{sec:discrete.bgg} we introduce the discrete setting and recall the discrete counterpart of the BGG diagram~\eqref{eq:double.complex.cont}. The analytical properties of the discrete diagram are established in Section~\ref{sec:analytical_properties}. Finally, in Section~\ref{sec:kirchhoff} we apply the Hodge Laplacian operator to the discrete Hessian complex in order to derive a discrete scheme for the Kirchhoff--Love problem and analyse its convergence. Three appendices conclude the article. In Appendix \ref{appendix:abstract.poincare} we develop a generic framework (used in Section \ref{sec:poincare}) for transferring Poincaré inequalities from one complex to another. Appendix \ref{appendix:P.nabla.k+3.T} discusses a potential reconstruction on the discrete $H_2$-space of the Stokes complex, which is of higher primal consistency degree than the one in Section \ref{sec:potential.XSgradT} but fails to achieve higher adjoint consistency degree. Finally, important notations used throughout the article are gathered in Appendix \ref{appendix:notations} for ease of reference.


\section{The Hodge Laplacian operator of the Hessian complex}\label{sec:hessian.complex}

Consider an abstract Hilbert complex, together with its adjoint operators
\[
\begin{tikzcd}[ampersand replacement=\&, column sep=4em]
  X_{k-1}  \arrow[r, black, "D^{k-1}", shift=({0,0.3em})] \& X_k \arrow[r, black, "D^{k}", shift=({0,0.3em})] \arrow[l, black, "D^{\adj}_{k-1}", shift=({0,-0.3em})] \& X_{k+1} \arrow[l, black, "D^{\adj}_{k}",shift=({0,-0.3em})].
\end{tikzcd}
\]
Note that the operators $D^\ell$ in this complex are closed and densely defined, which gives meaning to their adjoints $D^\adj_\ell$ for the ($L^2$-like) inner products $\langle\cdot,\cdot\rangle_{X_\ell}$ on the spaces $X_\ell$; see \cite[Section 4.1]{Arnold:18} for details.

The Hodge Laplacian on $X_k$ is defined by $\mathcal{L} \coloneqq D^{\adj}_{k}D^{k}+ D^{k-1}D^{\adj}_{k-1}$ and, for a source term $F\in (X_k)'$, the corresponding weak Hodge--Laplace problem reads: Find $u \in X_k$ such that
\begin{equation*} 
  \langle D^{k}u , D^{k}q \rangle_{X_{k+1}} + \langle D^{\adj}_{k-1}u, D^{\adj}_{k-1}q\rangle_{X_{k-1}} = F(q)\qquad\forall q \in X_k.
\end{equation*}
Dealing with adjoint operators would require to explicitly define their domain, which embeds boundary conditions; our discussion will be easier, from the PDE point of view, if we consider a mixed formulation of the Hodge--Laplace problem, consisting in introducing the additional variable $\omega=D_{k-1}^\adj u$: Find $(u,\omega) \in X_k\times X_{k-1}$ such that
\begin{equation} \label{eq:hodge.laplace}
  \begin{aligned}
    \langle D^{k}u , D^{k}q \rangle_{X_{k+1}} + \langle D^{k-1}\omega, q\rangle_{X_k} &= F(q)&&\quad\forall q \in X_k,\\
    \langle \omega, \mu\rangle_{X_{k-1}} - \langle u, D^{k-1}\mu\rangle_{X_k}&=0&&\quad\forall \mu\in X_{k-1}.
  \end{aligned}
\end{equation}

We can now describe the Hodge--Laplace problem associated with the first part of the Hessian complex \eqref{BGG-Hessian-HD}, when all the spaces are equipped with the $L^2$-inner products: 
For $X_k=H_2(\Omega)$, we have $X_{k-1}=\{0\}$ and $D^k=\hess$, so \eqref{eq:hodge.laplace} reads: Find $u\in H_2(\Omega)$ such that
\begin{equation*} 
  (\hess u, \hess q)_{L^2(\Omega)^{2\times 2}}=F(q)\qquad\forall q\in H_2(\Omega).
\end{equation*}
This is the standard weak formulation of the Kirchhoff--Love problem \eqref{eq:KL} (for $F(q)=(f,q)_{L^2(\Omega)}$) with natural boundary conditions \eqref{eq:KL:bc:natural}.

We can also consider the Hessian complex with homogeneous boundary conditions. Defining the following subspaces of $H_2(\Omega)$ and $ H_{\VROT}(\Omega, \mathbb{S})$:
\[
\begin{aligned}
  H_{2,0}(\Omega)
  & \coloneq \left\{
  u\in H_2(\Omega) \,:\, u = \nabla u \cdot\normal_{\Omega} =0  \text{ on } \partial\Omega
  \right\},
  \\
 \boldsymbol H_{\VROT,0}(\Omega, \mathbb{S}) &\coloneq
  \left\{
  \bvec{\xi} \in \boldsymbol H_{\VROT}(\Omega, \mathbb{S}) \,:\, \bvec \xi \, \tangent_{\Omega} = \bvec{0} \text{ on } \partial\Omega
  \right\},
\end{aligned}
\]
the Hessian complex with homogeneous boundary conditions reads
\begin{equation}\label{eq:bgg.complex.dirichlet}
  \begin{tikzcd}
    0 \arrow{r}
    & H_{2,0}(\Omega) \arrow{r}{\hess}
    & \boldsymbol H_{\VROT,0}(\Omega, \mathbb{S}) \arrow{r}{\VROT}
    & L^{2}(\Omega)^{2} \arrow{r}
    & 0.
  \end{tikzcd}
\end{equation}
The associated Hodge--Laplace problem is the same as above, but with $H_2(\Omega)$ replaced by $H_{2,0}(\Omega)$: Find $u\in H_{2,0}(\Omega)$ such that
\begin{equation} \label{eq:KL:weak:essential}
  (\hess u, \hess q)_{L^2(\Omega)^{2\times 2}}=F(q)\qquad\forall q\in H_{2,0}(\Omega),
\end{equation}
which corresponds to the Kirchhoff--Love problem \eqref{eq:KL} with essential (clamped) boundary conditions \eqref{eq:KL:bc:essential}.


\section{Discrete setting}\label{sec:discrete.bgg}

In this section we establish the discrete setting (mesh, polynomial spaces) and briefly recall the discrete counterpart of the BGG diagram \eqref{eq:double.complex.cont} introduced in \cite{Di-Pietro.Droniou.ea:26}. We then introduce the potential reconstructions and discrete $L^2$-products on the spaces of the diagram.

\subsection{Mesh and inequalities up to a constant}\label{sec:setting}

We denote by $\mathcal{M}_h = (\Th, \Eh, \Vh)$ a polygonal mesh of $\Omega$, where $\Th$ denotes a finite family of non-overlapping open polygonal elements $T$ with diameter $h_T$, such that $\overline{\Omega} = \bigcup_{T \in \Th} \overline{T}$; the mesh size is $h \coloneqq \max_{T \in \Th} h_T > 0$. The set $\Eh$ gathers a finite collection of open straight edges $E$, each of length $h_E$, while $\Vh$ denotes the set of vertices $V$ (with position vectors $\bvec{x}_V$), corresponding to the endpoints of edges in $\Eh$.

We assume that the pair $(\Th,\Eh)$ satisfies the compatibility conditions of \cite[Definition~1.4]{Di-Pietro.Droniou:20}. In particular, every edge lies on the boundary of at least one element, and the boundary of each element $T \in \Th$ is obtained as the union of the closures of the edges in the set $\ET$.

We also introduce the following notations or objects:
\begin{itemize}[itemsep=0em,leftmargin=1.5em]
\item $\mathcal{V}_Y$: set of vertices of $Y \in \Th \cup \Eh$.
\item $\tangent_E$: fixed tangent vector to $E\in\Eh$, determining its orientation.
\item $\normal_E$: normal vector to $E\in\Eh$ such that $(\tangent_E,\normal_E)$ is positively oriented.
\item $\omega_{TE}\in\{-1, 1\}$: orientation of $E\in\ET$ relative to $T\in\Th$, such that $\omega_{TE}\normal_E$ points outside $T$.
\item $\omega_{EV}\in\{-1,1\}$: orientation of $V\in\mathcal{V}_E$ relative to $E\in\Eh$, such that $\omega_{EV}\tangent_E$ points towards $V$.
\end{itemize}

Moreover, throughout the rest of the paper, the notation $a \lesssim b$ means that $a \le Cb$, where the constant $C$ depends only on $\Omega$, the mesh regularity parameter (see \cite[Assumption 7.6]{Di-Pietro.Droniou:20}), and, when polynomial functions are involved, the corresponding polynomial degree. The notation $a \simeq b$ means ``$a \lesssim b$ and $b \lesssim a$''.

\subsection{Polynomial spaces}

For any $Y\in\Th\cup\Eh$, we denote by $\Poly{\ell}(Y)$ the restriction to $Y$ of bivariate polynomials of total degree at most $\ell$, and set $\Poly{\ell}(Y) \coloneqq \{0\}$ when $\ell\le -1$. We also consider the zero-average subspace $\Poly{0,\ell}(Y)\subset\Poly{\ell}(Y)$.
The $L^2$-orthogonal projector on $\Poly{\ell}(Y)$ is denoted by $\lproj{\ell}{Y}:L^2(Y)\to\Poly{\ell}(Y)$, and we use the bold notation $\vlproj{\ell}{Y}$ for both its vector- and tensor-valued counterparts obtained applying $\lproj{\ell}{Y}$ component-wise.

For $\bullet\in\{T,h\}$, we denote by $\Poly{\ell}_{\rm c}(\mathcal{E}_\bullet)$ the set of functions that are continuous on $\bigcup_{E\in\mathcal{E}_\bullet}\overline{E}$ and polynomial of degree at most $\ell$ on each edge $E\in\mathcal{E}_\bullet$.
The broken polynomial space over the mesh is, on the other hand, denoted by $\Poly{\ell}(\Th)$ (no subscript ``$c$'').

The rotated gradient of a smooth scalar-valued function $q$ is $\CURL q \coloneq (\partial_2 q,-\partial_1 q)^\top$. It
satisfies, for any $T\in\Th$, any $E\in\ET$, and any $r\in\Cspace{1}(\overline{T})$,
\begin{equation}\label{eq:curl.normal}
  (\CURL r)_{|E}\cdot\normal_E = -\partial_{\tangent_E} r_{|E}.
\end{equation}
Finally, for each element $T\in\Th$, we select a point $\bvec{x}_T\in T$ such that a ball of diameter $\simeq h_T$ is contained in $T$. This choice allows us to introduce, for any $\ell\ge 0$, the subspaces of $\Poly{\ell}(T)^2$
\[
\Roly{\ell}(T) \coloneq \CURL\Poly{\ell+1}(T),
\qquad
\cRoly{\ell}(T) \coloneq (\bvec{x}-\bvec{x}_T)\Poly{\ell-1}(T).
\]
These spaces provide the (in general, non-orthogonal) decomposition
\begin{equation}\label{eq:decomp.poly2}
  \Poly{\ell}(T)^2 = \Roly{\ell}(T)\oplus \cRoly{\ell}(T).
\end{equation}

\subsection{Discrete BGG diagram}\label{sec:discrete.bgg.diagram}

Let an integer $k \ge 0$ be fixed.
The following diagram is the discrete counterpart of \eqref{eq:double.complex.cont} designed in  \cite{Di-Pietro.Droniou.ea:26}:
\begin{equation}\label{eq:double.complex}
  \begin{tikzcd}[column sep=2.5em]
    \text{DS($k$):}
    & 0 \arrow{r}{}
    & \XSgrad{\Th}\arrow{r}{\SGRAD{h}}
    & \XSrot{\Th}\arrow[dashed]{r}{\SROT{h}}
    & \Poly{k}(\Th)\arrow[dashed]{r}{}
    & 0 \\
    \text{DDR($k+1$):}
    & 0 \arrow[dashed]{r}{}
    & \tXdRgrad{\Th}\arrow{r}{\tGRAD{h}}\arrow[leftrightarrow]{ur}{\Id}
    & \tXdRrot{\Th}\arrow{r}{\tROT{h}}\arrow[dashed]{ur}{\sskwh}
    & \Poly{k+1}(\Th)^2\arrow{r}{}
    & 0.
  \end{tikzcd}
\end{equation}
The exponents involving $k$ refer to the polynomial consistency of the operators according to the convention established in \cite[Section~4]{Di-Pietro.Droniou.ea:26}.
The bottom row of \eqref{eq:double.complex} corresponds to the tensorisation of a version of the two-dimensional serendipity discrete de Rham (DDR) complex of \cite{Di-Pietro.Droniou:23*1} of degree $k + 1$ with, for all $T\in\Th$, the choice $\eta_T=3$ in \cite[Assumption 12]{Di-Pietro.Droniou:23*1}), while the top row is the Stokes complex designed in \cite[Section~4.2]{Di-Pietro.Droniou.ea:26}.

The spaces appearing in \eqref{eq:double.complex} are as follows:
\begin{align}\label{eq:def.Xgrad}
  \tXdRgrad{\Th} &\coloneq
  \begin{aligned}[t]
    \Big\{
    &\underline{\bvec{v}}_h=((\bvec{v}_T)_{T\in\Th},(\bvec{v}_E)_{E\in\Eh},(\bvec{v}_V)_{V\in\Vh})\,:\,
    \\
    &
    \bvec{v}_T\in\Poly{k-1}(T)^2 \quad \forall T\in\Th,\,
    \bvec{v}_E\in\Poly{k}(E)^2 \quad \forall E\in\Eh,\,
    \bvec{v}_V\in\Real^2 \quad \forall V\in\Vh
    \Big\},
  \end{aligned}
  \\ \label{eq:def.Xrot}
  \tXdRrot{\Th} &\coloneq
  \begin{aligned}[t]
    \Big\{
    &\underline{\bvec{\tau}}_h=((\bvec{\tau}_T)_{T\in\Th},(\bvec{\tau}_E)_{E\in\Eh})\,:
    \\
    &\bvec{\tau}_T\in\Poly{k}(T)^{2\times 2}\quad\forall T\in\Th,\,
    \bvec{\tau}_E\in\Poly{k+1}(E)^2\quad\forall E\in\Eh\Big\},
  \end{aligned}
  \\ \label{eq:def.XSgrad}
  \XSgrad{\Th} &\coloneq
  \begin{aligned}[t]
    \Big\{
    &\underline{q}_h=((q_T)_{T\in\Th},(q_E)_{E\in\Eh},(\Gqen)_{E\in\Eh},(q_V)_{V\in\Vh},(\Gqv)_{V\in\Vh})\,:
    \\
    &q_T\in\Poly{k-2}(T)\quad\forall T\in\Th,\,
    q_E\in\Poly{k-1}(E)\text{ and }\Gqen\in\Poly{k}(E)\quad\forall E\in\Eh,\,
    \\
    &q_V\in\Real\text{ and }\Gqv\in\Real^2\quad\forall V\in\Vh\Big\}.
  \end{aligned}
\end{align}

In what follows, the restriction of a discrete space to a mesh element or edge $Y \in \Th \cup \Eh$ (obtained collecting the components attached to $Y$ and its boundary) is denoted by replacing $\Th$ with $Y$; so, for example, we write 
\[
\XSgrad{E}\coloneq
\begin{aligned}[t]
  \Big\{
  &\underline{q}_E=(q_E,\Gqen,(q_V)_{V\in\mathcal V_E},(\Gqv)_{V\in\mathcal V_E})\,:
  \\
  &q_E\in\Poly{k-1}(E)\text{ and }\Gqen\in\Poly{k}(E),\,
  q_V\in\Real\text{ and }\Gqv\in\Real^2\quad\forall V\in\mathcal V_E\Big\}.
\end{aligned}
\]
As already done in the expression above, the restriction of a vector to $Y\in\Th\cup\Eh$ is denoted by replacing the index $h$ with $Y$.
  \begin{remark}[Comparison with the $C^1$-conforming space of \cite{Brezzi.Marini:13}]\label{rem:vem.comparison}
    In \cite{Brezzi.Marini:13}, the authors propose a $C^1$-conforming VEM space which corresponds to the following fully discrete space:
    For any polynomial degree $k \ge 2$,
    \begin{equation}\label{eq:C1}
      \underline{C}_1^k(\Th)
      \coloneq
      \begin{aligned}[t]
        \Big\{
        &\underline{q}_h=((q_T)_{T\in\Th},(q_E)_{E\in\Eh},(\Gqen)_{E\in\Eh},(q_V)_{V\in\Vh},(\Gqv)_{V\in\Vh})\,:
        \\
        &q_T\in\Poly{k-2}(T)\quad\forall T\in\Th\,,\quad
        q_E\in\Poly{k-2}(E)\text{ and }\Gqen\in\Poly{k-1}(E)\quad\forall E\in\Eh\,,
        \\
        &q_V\in\Real\text{ and }\Gqv\in\Real^2\quad\forall V\in\Vh\Big\}.
      \end{aligned}
    \end{equation}
    As for $\XSgrad{\Th}$, we have adopted the convention that the degree of the space is dictated by the degree of the Hessian reconstruction (which corresponds to a shift of $-2$ with respect to the meaning of $k$ in the above reference).
    The only difference between \eqref{eq:def.XSgrad} and \eqref{eq:C1} lies in the polynomial degree of the degrees of freedom corresponding to the moments of the function and of its normal derivative on edges.
    From this space, the authors of \cite{Brezzi.Marini:13} reconstruct a potential of degree $k+2$, leading to a discrete Hessian of degree $k$.
    As we will see, the richer boundary representation in \eqref{eq:def.XSgrad} makes it possible to reconstruct a Hessian one degree higher; see \eqref{eq:pol.const.hess} below. We note also that, contrary to the one in \cite{Brezzi.Marini:13}, our construction starts at $k=0$, and thus also provides lower-order (and less expensive) schemes.
\end{remark}

The interpolators on discrete spaces provide the interpretation of the vector of polynomials representing a smooth enough function. For the spaces defined by \eqref{eq:def.Xgrad}, \eqref{eq:def.Xrot} and \eqref{eq:def.XSgrad}, the interpolators are $\IdRgrad:\Cspace{0}(\overline{\Omega})^2\to \tXdRgrad{\Th}$, $\IdRrot:\Cspace{0}(\overline{\Omega})^{2\times 2}\to \tXdRrot{\Th}$ and $\ISgrad:\Cspace{1}(\overline{\Omega})\to \XSgrad{\Th}$ such that
\begin{alignat}{2}
  \label{eq:def.IdRgrad}
  \IdRgrad \bvec{v}\coloneq {}&((\vlproj{k-1}{T}\bvec{v})_{T\in\Th},(\vlproj{k}{E}\bvec{v})_{E\in\Eh},(\bvec{v}(\bvec{x}_V))_{V\in\Vh})&&\qquad\forall \bvec{v}\in \Cspace{0}(\overline{\Omega})^2,\\
  \label{eq:def.IdRrot}
  \IdRrot \bvec{\tau}\coloneq {}&((\vlproj{k}{T}\bvec{\tau})_{T\in\Th},(\vlproj{k+1}{E}(\bvec{\tau}\tangent_E))_{E\in\Eh})&&\qquad\forall \bvec{\tau}\in \Cspace{0}(\overline{\Omega})^{2\times 2},\\
  \label{eq:def.ISgrad}
  \ISgrad q\coloneq\Big( {}&(\lproj{k-2}{T}q)_{T\in\Th},(\lproj{k-1}{E}q)_{E\in\Eh},(\lproj{k}{E}(\GRAD q\cdot\normal_E))_{E\in\Eh},&&
  \\
    {}&(q(\bvec{x}_V))_{V\in\Vh},(\GRAD q(\bvec{x}_V))_{V\in\Vh}
    \Big) && \qquad\forall q\in \Cspace{1}(\overline{\Omega}).
\end{alignat}

We next recall the definitions of the operators of the DS($k$) complex and refer to \cite{Di-Pietro.Droniou.ea:26} for those of DDR($k+1)^2$.
The components of the discrete gradient and rotor for the DS complex are obtained mimicking appropriate integration by parts formulas as described below.

For all $E \in\Eh$ and all $\dof{q}{E}\in\XSgrad{E}$, the discrete tangential gradient $\Gqet \dof{q}{E} \in \Poly{k}(E)$ is such that 
\begin{equation}
  \label{eq:def.Gqet}
  \int_E \Gqet\underline{q}_E\,r=-\int_E q_E \partial_{\tangent_E} r + \sum_{V\in\VE}\omega_{EV}\,q_V\,r(\bvec{x}_V)\qquad\forall r\in\Poly{k}(E).
\end{equation}
For all $T\in\Th$ and all $\dof{q}{T}\in\XSgrad{T}$, the discrete element gradient $\nablaT\underline{q}_T\in\Poly{k-1}(T)^2$ satisfies
\begin{equation}
  \label{def:nablaT}
  \int_T\nablaT\underline{q}_T\cdot \bvec w =
  -\int_T q_T\DIV \bvec w
  + \sum_{E\in\ET}\omega_{TE}\int_E q_E\, (\bvec w\cdot\normal_{E})\qquad
  \forall \bvec w\in\Poly{k-1}(T)^2.
\end{equation}
The discrete gradient of $\underline{q}_h\in\XSgrad{\Th}$ is then given by
\begin{equation}\label{eq:def.nablah}
  \SGRAD{h}\dofh{q}=((\nablaT\underline{q}_T)_{T\in\Th},(\Gqet\dof{q}{E}\tangent_E + \Gqen \normal_E))_{E\in\Eh},(\Gqv)_{V\in\Vh})
  \in \XSrot{\Th}.
\end{equation}

For all $\underline{\bvec v}_h\in\tXdRgrad{\Th}$, $\SROT{h}\dof{\bvec{v}}{h}\in\Poly{k}(\Th)$ is such that
\begin{equation*} 
  (\SROT{h}\underline{\bvec v}_h)_{|T}
  = \SROT{T}\dof{\bvec v}{T}
  \qquad \forall T \in \Th,
\end{equation*}
where $\SROT{T}\dof{\bvec v}{T} \in\Poly{k}(T)$ satisfies
\begin{equation}\label{eq:def.SROT}
  \int_T\SROT{T}\dof{v}{T}\,r=\int_T \bvec{v}_T\cdot\CURL r - \sum_{E\in\ET}\omega_{TE}\int_{E}(\bvec{v}_E\cdot \tangent_E)\,r\qquad
  \forall r\in\Poly{k}(T).
\end{equation}

The following commutation properties, established in \cite[Lemma 11]{Di-Pietro.Droniou.ea:26}, will be useful.

\begin{lemma}[Commutation properties]\label{lem:com.prop}
  It holds, for all $T\in\Th$,
  \begin{alignat}{2}
    \label{eq:commutation.grad}
    \SGRAD{T} \ISgradT q {}&=\IdRgradT (\GRAD q)  &&\qquad\forall q \in \Cspace{1}(\overline{T}),\\
    \label{eq:commutation.rot}
    \SROT{T} \IdRgradT \bvec v {}&= \lproj{k}{T} (\ROT \bvec v)  &&\qquad\forall \bvec v \in H_2(T)^2.
  \end{alignat}
\end{lemma}

Following the solid arrows in diagram \eqref{eq:double.complex} from top left to bottom right, we obtain the following discrete counterpart of \eqref{BGG-Hessian-HD}:
\begin{equation}\label{derived-BGG}
  \begin{tikzcd}
    \text{DH($k+1$):}
    &[-.5em] 0 \arrow{r}{}
    &[-.5em] \XSgrad{\Th} \arrow{r}{\Hess{h}}
    &[1.5em] \XdRrotS{\Th} \arrow{r}{\tROT{h}}
    &[1.5em] \Poly{k+1}(\Th)^2 \arrow{r}{}
    &[-.5em] 0
  \end{tikzcd}
\end{equation}
where $\sskwh:\tXdRrot{\Th}\to\Poly{k}(\Th)$ is such that $(\sskwh \underline{\bvec{\tau}}_h)_{|T} \coloneq \sskw \bvec{\tau}_T$ for all $T \in \Th$, the space in the middle of \eqref{derived-BGG} is given by
\[
\begin{aligned}
  \XdRrotS{\Th}
  &\coloneq \tXdRrot{\Th}\cap \Ker(\sskwh)\\
  &=
  \Big\{
  \underline{\bvec{\tau}}_h=((\bvec{\tau}_T)_{T\in\Th},(\bvec{\tau}_E)_{E\in\Eh})\,:\,
  \bvec{\tau}_T\in\Poly{k}(T)^{2\times 2}\text{ with $\bvec{\tau}(\bvec{x})\in\mathbb{S}$ $\forall \bvec{x}\in T$\quad$\forall T\in\Th$},\\
  &\qquad\bvec{\tau}_E\in\Poly{k+1}(E)^2\quad\forall E\in\Eh
  \Big\},
\end{aligned}
\]
and
\begin{equation*} 
  \Hess{h} \coloneq \tGRAD{h}\circ \SGRAD{h}
\end{equation*}
is the discrete Hessian operator.
The degree of this operator is justified by its polynomial consistency stated in the next lemma (see \cite[Lemma 20]{Di-Pietro.Droniou.ea:26} for the proof).

\begin{lemma}[Commutation property and polynomial consistency of the discrete Hessian operator]
  \label{lem:pol.const.hess}
  For all $T\in\Th$, it holds
  \begin{equation}\label{eq:com.prop.hess}
    \Hess{T}\ISgradT q =\IdRrotT \hess q \qquad \forall q \in \Cspace{2}(T),
  \end{equation}
  and the operator $\Hess{T}$ is consistent of degree $k+1$, that is,
  \begin{equation}\label{eq:pol.const.hess}
    \Prot{k+1}{T} \Hess{T}\ISgradT q = \hess q\qquad\forall q\in\Poly{k+3}(T),
  \end{equation}
  where $\Prot{k+1}{T}$ is the tensorised version of the potential on $\XdRrot{\Th}$ built from the two-dimensional tangential trace of \cite[Eq.~(3.22), (3.23)]{Di-Pietro.Droniou:23} and transferred to the serendipity complex via \cite[Section 2]{Di-Pietro.Droniou:23*1}.
\end{lemma}


\subsection{Serendipity operators}

We recall here the serendipity operators relevant to the following discussion.
These operators are designed to reduce the dimension of a discrete space while preserving its degree of polynomial consistency.

  Let $T\in\Th$. With the choice of serendipity discussed in Section \ref{sec:discrete.bgg.diagram}, the local and non-tensorized DDR $\boldsymbol{H}_{\VROT}$-space of degree $k$ is
  \[
  \XdRrotk{T} \coloneq  \Big\{
  \underline{\bvec{v}}_h=((\bvec{v}_T)_{T\in\Th},(v_E)_{E\in\Eh})\,:
  \\
  \bvec{v}_T\in\Poly{k-1}(T)^{2},\,
  v_E\in\Poly{k}(E)\quad\forall E\in\Eh\Big\},
  \]
  with interpolator $\underline{I}_{\ROT,T}^k\bvec{v} \coloneqq (\lproj{k-1}{T}\bvec{v},(\lproj{k}{E}(\bvec{v}\cdot\tangent_E))_{E\in\ET})$ for all $\bvec{v}\in \mathcal C_0(T)^2$. Let us define on $\XdRrotk{T}$ a discrete (low-accuracy) rotor $\ROT^{k-1}_T: \XdRrotk{T}\to\Poly{k-1}(T)$ by mimicking an integration-by-parts formula (in a similar way as \eqref{eq:def.SROT}): for all $\dof{\bvec{v}}{T}\in  \XdRrotk{T}$,
  \[
  \int_T (\ROT^{k-1}_T\dof{\bvec{v}}{T})\,r = \int_T \bvec{v}_T\cdot\CURL r - \sum_{E\in\ET}\omega_{TE}\int_E v_E r\qquad\forall r\in\Poly{k-1}(T).
  \]
  It is a simple matter to check that
  \begin{equation}\label{eq:polynomial.consistency.ROT_T}
    \ROT^{k-1}_T \underline{I}_{\ROT,T}^k\bvec{v}=\ROT \bvec{v}\quad\forall \bvec{v}\in\Poly{k}(T)^2.
  \end{equation}
  The curl serendipity operator $\SC{T}{k}:\XdRrotk{T} \to \Poly{k}(T)^2$ is then defined such that, for all $\dofT{\bvec{v}} \in \XdRrotk{T}$, $\SC{T}{k}\dofT{\bvec{v}}$ is the solution of the following constrained minimisation problem:
  \begin{equation} \label{eq:min.Jv}
    \text{Minimize $\mathcal{J}(\dof{\bvec{v}}{T}; \bvec{w})$ over the set of $\bvec{w}\in\Poly{k}(T)^2$ such that $\Rcproj{k-1}{T}\bvec{w}=\Rcproj{k-1}{T}\bvec{v}_T$},
  \end{equation}
  where $\Rcproj{k-1}{T}$ is the $L^2(T)$-orthogonal projection on $\cRoly{k-1}(T)$ and, for all $\bvec{w}\in\Poly{k}(T)^2$,
  \begin{align*}
   \mathcal{J}(\dof{\bvec{v}}{T}; \bvec{w}) \coloneq \sup_{\bvec{q}\in B(\Poly{k}(T)^2)}{}& \left\{
    h_T\sum_{E\in\ET}\int_E(\bvec{w}\cdot\tangent_E-v_E)\,(\bvec{q}\cdot\tangent_E)
    + h_T^2\left[\int_T(\ROT\bvec{w}-\ROT_T^{k-1}\dof{\bvec{v}}{T})\ROT\bvec{q}
      \right]
    \right\}
  \end{align*}
  with $B(\Poly{k}(T)^2)\coloneq\{\bvec{q}\in\Poly{k}(T)^2\,:\,\norm{L^2(T)}{\bvec{q}}\le 1\}$ the unit ball in $\Poly{k}(T)^2$.
  The well-posedness of the minimisation problem \eqref{eq:min.Jv} is a consequence of \cite[Proposition 16]{Di-Pietro.Droniou:23*1}.
  In this problem, the role of the first term in the supremum defining $\mathcal{J}(\dof{\bvec{v}}{T}; \cdot)$ is to impose, in a least-square sense, the relation $\bvec{w}\cdot\tangent_E=v_E$ on all $E\in\ET$, while the second term imposes, still in the least-square sense, the condition $\ROT\bvec{w}=\ROT_T^{k-1}\dof{\bvec{v}}{T}$. Owing to \eqref{eq:polynomial.consistency.ROT_T} we have, whenever $\bvec{v}\in\Poly{k}(T)^2$,
  \begin{align*}
    \mathcal{J}(\underline{I}_{\ROT,T}^k\bvec{v}; \bvec{w})= \sup_{\bvec{q}\in B(\Poly{k}(T)^2)}{}& \Bigg\{h_T\sum_{E\in\ET}\int_E(\bvec{w}- \bvec{v})\cdot\tangent_E\,(\bvec{q}\cdot\tangent_E)
    + h_T^2\left[\int_T(\ROT\bvec{w}-\ROT\bvec{v})\,\ROT\bvec{q}\right]\Bigg\}.
  \end{align*}
  It is then obvious that $\SC{T}{k}$ exactly reconstructs polynomials of degree $\le k$ from the information provided by their interpolate on $\XdRrotk{T}$ (see also \cite[Eq.~(6.3)]{Di-Pietro.Droniou:23*1}):
  \begin{equation}\label{eq:consistency.SC}
    \SC{T}{k}\underline{I}_{\ROT,T}^k\bvec{v}=\bvec{v}\qquad\forall \bvec{v}\in\Poly{k}(T)^2.
  \end{equation}
  
  To use this serendipity operator on $\XSrot{T}$, we define the injection $\injrot :\XSrot{T}\rightarrow \XdRrotk{T}$ such that
  \begin{equation*}
    \injrot(\dofT{\bvec v})  = (\bvec{v}_T,(\bvec{v}_E\cdot \tangent_E)_{E\in\ET})
    \qquad \forall \dofT{\bvec v}\in\XSrot{T}.
  \end{equation*}
  Letting $\SR{T}{k} \coloneq \SC{T}{k}\circ\injrot:\XSrot{T}\to\Poly{k}(T)^2$ and noticing that $\injrot\circ \IdRgradT =\underline{I}_{\ROT,T}^k$, the polynomial consistency \eqref{eq:consistency.SC} shows that
  \begin{equation}
    \label{eq:pol.const.seren.Srot}
    \SR{T}{k} \IdRgradT \bvec v = \bvec v \qquad \forall \bvec v \in \Poly{k}(T)^2.
  \end{equation}

\subsection{Potential reconstruction on $\XSrot{T}$}

For $T\in\Th$, two discrete calculus operators are defined on $\XSrot{T}$ ($\tGRAD{T}$ and $\SROT{T}$, see \eqref{eq:double.complex}), leading to two distinct potential reconstructions on this space.
The potential reconstruction $\Pgrad{k+2}{T}:\XSrot{T}\rightarrow\Poly{k+2}(T)^2$ associated with $\tGRAD{T}$ is the tensorised standard serendipity potential reconstruction of degree $k+2$, defined in \cite[Section 2 and 4.2.1]{Di-Pietro.Droniou:23*1}.
The potential reconstruction $\SProt:\XSrot{T}\rightarrow \Poly{k}(T)^2$ associated with $\SROT{T}$ is, on the other hand, such that, for all $\dof{\bvec v}{T}\in\XSrot{T}$,
\begin{multline}\label{eq:def.potSrot}
  \int_T\SProt\dof{\bvec v}{T}\cdot(\CURL r+\bvec w)
  = \int_T\SROT{T}\dof{\bvec{v}}{T}\, r
  + \sum_{E\in\ET}\omega_{TE}\int_{E}(\bvec{v}_E\cdot \tangent_E) \,r
  + \int_T \SR{T}{k}\dof{\bvec v}{T} \cdot \bvec w \\
  \forall (r,\bvec w) \in \mathcal{P}^{k+1}(T)\times\cRoly{k}(T).
\end{multline}
Notice that $\SProt\dof{\bvec v}{T}$ is well-defined owing to the decomposition \eqref{eq:decomp.poly2} and since the right-hand side of \eqref{eq:def.potSrot} vanishes when applied to $r$ such that $\CURL r=0$, by definition \eqref{eq:def.SROT} of $\SROT{T}$ and using $\Ker \CURL = \Poly{0}(T)$.

The potentials on $\XSrot{T}$ are both polynomially consistent at their respective degrees, i.e,
\begin{alignat}{2}
  \label{eq:const.pol.P.grad.k+2}
  \Pgrad{k+2 }{T} \IdRgradT \bvec v&{}= \bvec v &&\qquad \forall \bvec v \in\Poly{k+2}(T)^2,\\
  \label{eq:const.pol.P.rot.k}
  \SProt\IdRgradT \bvec w &{}= \bvec w &&\qquad \forall \bvec w\in\Poly{k}(T)^2.
\end{alignat}
The result on $\Pgrad{k+2}{T}$ comes from the serendipity DDR framework, while \eqref{eq:const.pol.P.rot.k} is a consequence of \eqref{eq:commutation.rot} and \eqref{eq:pol.const.seren.Srot}. Concerning the polynomial degree for $\SProt$, remember that, regarded as the domain of $\SROT{T}$, the space $\XSrot{T}$ is only polynomially consistent of degree $k$ (see Lemma \ref{lem:com.prop}).

\subsection{Potential reconstruction on $\XSgrad{T}$}\label{sec:potential.XSgradT}

Let $T\in\Th$. Let us first design a trace reconstruction on $\partial T$. For all $\dofT{q}\in\XSgrad{T}$, let $\gammaS\dof{q}{T}\in \Poly{k+1}_{\rm c}(\ET)$ be such that:
\begin{equation}\label{eq:def.gammaS}
  \lproj{k-1}{E}(\gammaS\dof{q}{T}) = q_E\quad\forall E\in\ET\,,\qquad \gammaS \dof{q}{T} (\bvec{x}_V)= q_V\quad\forall V\in\VT.
\end{equation}
Remark that the definition \eqref{eq:def.Gqet} of $\Gqet$ easily gives
\begin{equation}
  \label{eq:link.gammaS.Gqet}
  \partial_{\tangent_E}  (\gammaS \dof{q}{T})_{|E}=\Gqet \dof{q}{E}.
\end{equation}
We also notice that
\begin{equation}\label{eq:trace.exactness}
  \gammaS(\ISgradT q ) = q_{|\partial T}\qquad\forall q \in \Cspace{1}(\overline{T})\text{ such that }q_{|\partial T} \in \Poly{k+1}(\Eh).
\end{equation}
For $\dof{q}{T}\in\XSgrad{T}$, the potential reconstruction $\SPgrad\dof{q}{T}\in\Poly{k+1}(T)$ is such that
\begin{multline}
  \label{eq:def.SPGRAD}
  \int_T \SPgrad \dof{q}{T}\DIV \bvec w = -\int_T \SProt\SGRAD{T}\dof{q}{T}\cdot \bvec w+\sum_{E\in\ET}\omega_{TE}\int_E \gammaS \dofT{q}\, (\bvec w\cdot\normal_E)\\
  \forall \bvec w\in \cRoly{k+2}(T).
\end{multline}
This formula uniquely defines $\SPgrad \dof{q}{T}\in \Poly{k+1}(T)$ because $\DIV:\cRoly{k+2}(T)\to \Poly{k+1}(T)$ is an isomorphism.

\begin{remark}[Higher-order potential]
  The information available $\XSgrad{T}$ would actually allow us to reconstruct a potential that has primal consistency properties up to degree $k+3$.
  However, the degree of the adjoint consistency properties of this potential seems to only be $k+1$. We detail this in Appendix \ref{appendix:P.nabla.k+3.T}.
\end{remark}

\begin{remark}[Validity of \eqref{eq:def.SPGRAD}]\label{rem:extention.SPGRAD}
  Take $\bvec w \in\Roly{k}(T)$ in \eqref{eq:def.SPGRAD}. The left hand side vanishes because $\DIV \bvec w = 0$. Let us show that the right-hand side vanishes as well. Letting $r\in\Poly{k+1}(T)$ be such that $\CURL r = \bvec w$, we have
  \begin{equation}\label{eq:Prot.valid.w}
    \begin{aligned}
      \int_T \SProt\SGRAD{T}\dof{q}{T}\cdot \CURL r \overset{\eqref{eq:def.potSrot},\,\eqref{eq:def.nablah}}&{=} \int_T\cancel{\SROT{T}\SGRAD{T}\dofT{q}}\, r + \sum_{E\in\ET}\omega_{TE}\int_{E}\Gqet\dofT{q} \,r \\
      \overset{\eqref{eq:link.gammaS.Gqet},\,\text{IBP}}&{=} -\sum_{E\in\ET}\omega_{TE}\int_{E}\gammaS\dofT{q} \,\partial_{\tangent_E} r \\
      \overset{\eqref{eq:curl.normal}}&= \sum_{E\in\ET}\omega_{TE}\int_{E}\gammaS\dofT{q} \,(\bvec w \cdot\normal_E),
    \end{aligned}
  \end{equation}
  where the cancellation in the first equality comes from the complex property of DS($k$) \cite[Lemma 13]{Di-Pietro.Droniou.ea:26}.
  Hence, the formula \eqref{eq:def.SPGRAD} can be extended to all $\bvec w \in\Roly{k}(T)\oplus \cRoly{k+2}(T)$ (hence, in particular, to all $\bvec{w} \in \Poly{k}(T)^2$).
\end{remark}

The following polynomial consistency property for $\SPgrad$ is a direct consequence of its definition \eqref{eq:def.SPGRAD} together with the commutation property \eqref{eq:commutation.grad} and the polynomial consistency properties \eqref{eq:const.pol.P.rot.k} of $\SProt$ and \eqref{eq:trace.exactness} of $\gammaS$:
\begin{equation}\label{eq:poly.consistency.Pgrad}
  \SPgrad \ISgradT q = q \qquad \forall q\in \Poly{k+1}(T).
\end{equation}

\subsection{$L^2$-like norms and scalar products}

Throughout the rest of the paper, given an open bounded subset $Y$ of $\Real^2$, we denote by $\norm{Y}{{\cdot}}$ the standard norm of $L^2(Y)$, $L^2(Y)^2$, or $L^2(Y)^{2\times 2}$, all possible ambiguity being removed by the argument.
On $\XSgrad{\Th}$, we define the norm $\nSgrad{h}{{\cdot}}$ such that, for all $\dof{q}{h} \in \XSgrad{\Th}$,
\begin{equation}\label{eq:def.norm.gradT}
  \begin{gathered}
    \nSgrad{h}{\dof{q}{h}}^2 \coloneqq \sum_{T \in \Th} \nSgrad{T}{\dof{q}{T}}^2,
    \\
    \nSgrad{T}{\dof{q}{T}}^2
    \coloneq \norm{T}{q_T}^2
    + \sum_{E \in \ET}h_T \left(\norm{E}{q_E}^2
    + h_T^2\norm{E}{\Gqen}^2
    \right)
    + \sum_{V\in\VT}h_T^2 \left(
    |q_V|^2+h^2_T|\Gqv|^2
    \right).
  \end{gathered}
\end{equation}
The local norm $\nSgrad{T}{{\cdot}}$ is induced by a natural inner product that we denote by $\langle \cdot ,\cdot\rangle_{\stokes,T}$.

The norm on $\XSrot{\Th}$ is such that, for all $\dofT{\bvec v}\in\XSrot{T}$,
\begin{equation}\label{eq:def.norm.rotT}
  \begin{gathered}
    \nSrot{h}{\dofh{\bvec v}}^2 \coloneq  \sum_{T\in\Th} \nSrot{T}{\dofT{\bvec v}}^2,
    \\
    \nSrot{T}{\dofT{\bvec v}}^2\coloneq
    \norm{T}{\bvec v_T}^2+ \sum_{E\in\ET} h_T \norm{E}{\bvec v_E}^2
    +\sum_{V\in\VT} h_T^2| \bvec v_V|^2
  \end{gathered}
\end{equation}
and the inner product inducing $\nSrot{T}{{\cdot}}$ is denoted by $ \langle \cdot,\cdot \rangle_{1,T}$.

For $T\in\Th$, the discrete $L^2$-scalar product $(\cdot,\cdot)_{\stokes,T}$  on $\XSgrad{T}$ is such that, for all $\dof{q}{T},\dof{r}{T}\in \XSgrad{T}$,
\begin{equation}\label{eq:prod.l2.2.T}
  \begin{gathered}
    (\dof{q}{T},\dof{r}{T})_{\stokes,T}\coloneq  \int_T \SPgrad \dof{q}{T}\,\SPgrad\dof{r}{T} +s_{\stokes,T}(\dof{q}{T},\dof{r}{T}),
    \\
    s_{\stokes,T}(\dof{q}{T},\dof{r}{T}) \coloneq  \langle \dof{q}{T} - \ISgradT \SPgrad\dof{q}{T} , \dof{r}{T} - \ISgradT \SPgrad\dof{r}{T} \rangle_{\stokes,T}.
  \end{gathered}
\end{equation}
The discrete $L^2$-scalar product  $(\cdot,\cdot)_{\ROT,T}$  on $\XSrot{T}$ is such that, for all $\dof{\bvec v}{T},\dof{\bvec w}{T}\in \XSrot{T}$,
\[
\begin{gathered}
  (\dof{\bvec v}{T},\dof{\bvec w}{T})_{\ROT,T}\coloneq
  \int_T\SProt\dofT{\bvec v} \cdot \SProt\dofT{\bvec w} + s_{\ROT,T}(\dof{\bvec v}{T},\dof{\bvec w}{T}),
  \\
  s_{\ROT,T} \coloneq   \langle \dofT{\bvec v} - \IdRgrad\SProt\dofT{\bvec v}, \dofT{\bvec w} - \IdRgrad\SProt\dofT{\bvec w} \rangle_{1,T}.
\end{gathered}
\]

For $\bullet\in\{\stokes,\ROT\}$, the norm induced by $(\cdot,\cdot)_{\bullet,T}$ is denoted by $\norm{\bullet,T}{{\cdot}}$. The corresponding global inner product $(\cdot,\cdot)_{\bullet,h}$ and norm $\norm{\bullet,h}{{\cdot}}$ are defined by assembling local contributions.

\begin{lemma}[Norms equivalence]
  The following uniform norm equivalences hold:
  \begin{equation*}
    \norm{\stokes,T}{{\cdot}}\simeq \nSgrad{T}{{\cdot}}\,,\qquad
    \norm{\ROT,T}{{\cdot}}\simeq \nSrot{T}{{\cdot}}.
  \end{equation*}
\end{lemma}
The proof of this equivalence follows the same reasoning as \cite[Lemma 5]{Di-Pietro.Droniou:23} and is based on appropriate boundedness properties of the potential reconstructions $\SPgrad$ and $\SProt$, obtained as in \cite[Proposition 13]{Di-Pietro.Droniou:21}.

\begin{proposition}[Boundedness of local operators of the DS($k$) complex]
  \label{prop:boundedness.local.operators}
  For all $T\in\Th$, it holds
  \begin{alignat*}{2}
    \norm{T}{\SPgrad\dof{q}{T}}+h_T\nSrot{T}{\SGRAD{T}\dof{q}{T}} &\lesssim \nSgrad{T}{\dof{q}{T}} &&\qquad \forall \dof{q}{T}\in\XSgrad{T},
    \\
    \norm{T}{\SProt\dof{\bvec v}{T}}
    + h_T \norm{T}{\SROT{T}\dof{\bvec v}{T}} &\lesssim \nSrot{T}{\dof{\bvec v}{T}} &&\qquad \forall \dof{\bvec v}{T}\in\XSrot{T}.
  \end{alignat*}
\end{proposition}


\section{Analytical properties of the discrete Stokes complex}\label{sec:analytical_properties}

In this section we discuss the analytical properties of the discrete Stokes complex in the top row of \eqref{eq:double.complex}.

\subsection{Primal and adjoint consistency}

The space $\bvec{H}_{\DIV,0}(\Omega)$ (resp.~$H_{1,0}(\Omega)$) is the subspace of $\bvec{H}_{\DIV}(\Omega)$ (resp.~$H_{1}(\Omega)$)  spanned by functions whose normal trace (resp.~trace) vanishes on the boundary of $\Omega$.
We denote by $H_k(\Th) \ni v \mapsto \seminorm{H_k(\Th)}{v} \coloneq \big( \sum_{T \in \Th} \seminorm{H_k(T)}{v}^2 \big)^{\nicefrac12} \in \Real$ the broken $H_k$-seminorm, with  $\seminorm{H_k(T)}{{\cdot}}$ denoting the standard seminorm of $H_k(T)$.
The semi-norm $\seminorm{H_{(k+1,2)}(\Th)^2}{{\cdot}}$ on the space $H_{\max(k+1,2)}(\Th)^2$ is defined by: for all $\bvec{v}\in H_{\max(k+1,2)}(\Th)^2$,
\begin{gather*}
  \seminorm{H_{(k+1,2)}(\Th)^2}{\bvec{v}}\coloneq \left(\sum_{T\in\Th}\seminorm{H_{(k+1,2)}(T)^2}{\bvec{v}}^2\right)^{\nicefrac12}
  \\
  \text{ with }\quad
  \seminorm{H_{(k+1,2)}(T)^2}{\bvec{v}}\coloneq\left\{
  \begin{array}{l@{\quad}l}
    \seminorm{H_1(T)^2}{\bvec{v}}+h_T\seminorm{H_2(T)^2}{\bvec{v}}&\mbox{ if $k=0$},\\
    \seminorm{H_{k+1}(T)^2}{\bvec{v}}&\mbox{ if $k\ge 1$}.
  \end{array}\right.
\end{gather*}
The proofs of the theorems stated below are omitted as they are similar to those of \cite[Section 6]{Di-Pietro.Droniou:23}, using the serendipity framework of \cite{Di-Pietro.Droniou:23*1} and, for \eqref{eq:adjoint.const.Sgrad}, invoking Remark \ref{rem:extention.SPGRAD}.
We just note that these proofs are based on Proposition \ref{prop:boundedness.local.operators},the polynomial consistency properties expressed by \eqref{eq:const.pol.P.rot.k} and \eqref{eq:poly.consistency.Pgrad}, and the two following boundedness properties of the interpolators (which can be established along the same lines as \cite[Lemma 6]{Di-Pietro.Droniou:23}): For all $T\in\Th$,
\begin{equation*}
  \begin{aligned}
    \nSgrad{T}{\ISgradT q} &\lesssim \sum_{i=0}^{2} h_T^{i} \seminorm{H_{i}(T)}{q} \qquad \forall q \in H_2(T),\\
    \nSrot{T}{\IdRgradT \bvec{v}} &\lesssim \sum_{i=0}^{2} h_T^{i} \seminorm{H_{i}(T)^2}{\bvec{v}} \qquad \forall \bvec{v} \in H_2(T)^2.
  \end{aligned}
\end{equation*}

\begin{theorem}[Consistency results on $\XSgrad{\Th}$]~
  \label{th:const.grad}
  \begin{enumerate}
  \item \emph{Consistency of the potential and gradient reconstructions}.    For all $T\in\Th$ and all $q\in H_{k+2}(T)$, it holds
    \begin{equation*}
      \norm{T}{\SPgrad(\ISgradT q) -q}
      + h_T \norm{T}{\SProt\SGRAD{T}  (\ISgradT q) -\GRAD q} \lesssim h^{k+2}_T\seminorm{H_{k+2}(T)}{q}.
    \end{equation*}
  \item \emph{Consistency of $(\cdot,\cdot)_{2,T}$}. It holds, for all $T\in\Th$ and all  $(r,\dofT{q})\in H_{k+2}(T)\times \XSgrad{T}$,
    \[
    \left|\int_T r \, \SPgrad\dofT{q} - (\ISgradT r ,\dofT{q})_{2,T}\right|\lesssim h^{k+2}_T \seminorm{H_{k+2}(T)}{r} \, \norm{2,T}{\dofT{q}}.
    \]
  \item \emph{Adjoint consistency of $\SGRAD{h}$}.
    Let $\bvec{\mathcal{D}} \coloneq \Cspace{0}(\overline{\Omega})^2\cap \bvec{H}_{\DIV,0}(\Omega)$ and define the adjoint consistency error associated with $\SGRAD{h}$ as the bilinear form $\acSgrad:\bvec{\mathcal{D}}\times \XSgrad{\Th}\rightarrow \Real$ such that, for all $(\bvec v, \dofh{q}) \in \bvec{\mathcal{D}}\times \XSgrad{\Th}$,
    \[
    \acSgrad(\bvec v,\dofh{q})\coloneq  \sum_{T\in\Th}\left[ (\IdRgradT \bvec{v}_{|T}, \SGRAD{T}\dof{q}{T})_{\ROT,T}+ \int_T \DIV \bvec{v}\,\SPgrad\dof{q}{T}\right].
    \]
    Then, for all $\bvec{v}\in\bvec{\mathcal{D}}\cap H_{\max(k+1,2)}(\Th)^2$ and all $\dofh{q}\in\XSgrad{\Th}$, it holds
    \begin{equation}
      \label{eq:adjoint.const.Sgrad}
      |\acSgrad(\bvec v,\dofh{q})|\lesssim
      h^{k+1}\seminorm{H_{(k+1,2)}(\Th)^2}{\bvec v} \, \norm{\ROT,h}{\SGRAD{h}\dofh{q}}.
    \end{equation}
  \end{enumerate}
\end{theorem}

\begin{remark}[Adjoint consistency using the $(\cdot,\cdot)_{1,T}$ product]
    A consistency result similar to \eqref{eq:adjoint.const.Sgrad} can be obtained for the adjoint consistency error obtained replacing $(\cdot,\cdot)_{\ROT,T}$ with $(\cdot,\cdot)_{1,T}$, the $L^2$-discrete scalar product on $\XSrot{T}$ obtained by tensorising the scalar product on $\underline{H}_{1}^{k+1}(T)$ defined in \cite[Eq.~(4.14)]{Di-Pietro.Droniou:23}.
\end{remark}

\begin{theorem}[Consistency results on $\XSrot{\Th}$]~
  \begin{enumerate}
  \item \emph{Consistency of the potential reconstruction}.
    For all $T\in\Th$, it holds
    \begin{equation*}
      \norm{T}{\SProt(\IdRgradT \bvec v) -\bvec v }\lesssim h^{k+1}_T \seminorm{H_{(k+1,2)}(T)^2}{\bvec v} \qquad \forall \bvec v \in H_{\max(k+1,2)}(T)^2.
    \end{equation*}
  \item \emph{Primal consistency of $\SROT{T}$.}
    For all $T\in\Th$ and all $\bvec v \in \Cspace{0}(\overline{T})^2$ such that $\ROT \bvec v \in H_{k+1}(T)$, it holds
    \begin{equation*}
      \norm{T}{\SROT{T} (\IdRgradT \bvec v) -\ROT \bvec v}\lesssim h^{k+1}_T\seminorm{H_{k+1}(T)}{\ROT \bvec{v}}.
    \end{equation*}
  \item \emph{Consistency of $(\cdot,\cdot)_{\ROT,T}$}. It holds, for all $T\in\Th$, and all  $(\bvec w ,\dofT{\bvec v})\in H_{\max(k+1,2)}(T)^2\times \XSrot{T}$,
    \begin{equation*}
      \left|\int_T \bvec w \, \SProt \dofT{\bvec v} - (\IdRgradT \bvec w ,\dofT{\bvec v})_{\ROT,T}\right|\lesssim h^{k+1}_T \seminorm{H_{(k+1,2)}(T)^2}{\bvec w} \norm{2,T}{\dofT{\bvec v}}.
    \end{equation*}
  \item \emph{Adjoint consistency of $\SROT{h}$}.
    Let $\mathcal{D} \coloneq \Cspace{0}(\overline{\Omega}) \cap H_{1,0}(\Omega)$ and define
    the adjoint consistency error associated to $\SROT{h}$ as the bilinear form $\acSrot: \mathcal{D} \times \XSrot{\Th}\to \Real$ such that, for all $(r,\dofh{\bvec v})\in \mathcal{D} \times \XSrot{\Th}$,
    \begin{equation*}
      \acSrot(r,\dofh{\bvec v}) \coloneq  \sum_{T\in\Th} \left[  \int_T \lproj{k}{T}r \, \SROT{T}\dofT{\bvec v} - \int_T \CURL r \cdot \SProt \dofT{\bvec v}\right].
    \end{equation*}
    Then, for all $r\in \mathcal{D}\cap H_{k+2}(\Th)$ and all $\dofh{\bvec v}\in\XSrot{\Th}$,
    \begin{equation*}
      |\acSrot(r,\dofh{\bvec v})| \lesssim h^{k+1}\seminorm{H_{k+2}(\Th)}{r}\norm{\ROT,h}{\dofh{\bvec v}}.
    \end{equation*}
  \end{enumerate}
\end{theorem}

\subsection{Poincaré--Wirtinger and Poincaré inequalities}\label{sec:poincare}

\subsubsection{Poincaré--Wirtinger inequalities}

We state here Poincaré inequalities for both the operators in the DS($k$) complex. The proofs of the following
theorems are given in Section \ref{sec:proof.poincares}, using the abstract setting developed in Appendix \ref{appendix:abstract.poincare}.

\begin{theorem}[Poincaré inequality on $\XSgrad{\Th}$]
  \label{th:pc.grad}
  Denoting by $(\Ker\SGRAD{h})^{\perp}$ the orthogonal complement in $\XSgrad{\Th}$ of $\Ker\SGRAD{h}$ for the inner product $(\cdot,\cdot)_{2,h}$, it holds
  \begin{equation*} 
    \nSgrad{h}{\dof{q}{h}} \lesssim \nSrot{h}{\SGRAD{h}\dof{q}{h}}\qquad
    \forall \dofh{q}\in(\Ker\SGRAD{h})^{\perp}.
  \end{equation*}
\end{theorem}

\begin{remark}[Equivalent formulation of the orthogonality condition]
  Using the consistency of the stabilisation component in \eqref{eq:prod.l2.2.T}, it  can be checked that $\dofh{q}\in(\Ker\SGRAD{h})^\perp$ is equivalent to $\sum_{T\in\Th}\int_T \SPgrad\dof{q}{T}=0$.
\end{remark}

\begin{theorem}[Poincaré inequality on $\XSrot{\Th}$]
  \label{th:pc.rot}
  Denoting by $(\Ker\SROT{h})^{\perp}$ the orthogonal complement in $\XSrot{\Th}$ of $\Ker\SROT{h}$ for the inner product $(\cdot,\cdot)_{\ROT,h}$, it holds
  \begin{equation*}
    \nSrot{h}{\dof{\bvec{v}}{h}} \lesssim \norm{\Omega}{\SROT{h}\dof{\bvec{v}}{h}} \qquad \forall \dof{\bvec{v}}{h}\in(\Ker\SROT{h})^{\perp}.
  \end{equation*}
\end{theorem}

\subsubsection{Preliminary Poincaré inequalities}
\label{sec:intermediate.poincare}

In this section we establish Poincaré inequalities on particular subspaces of the DS($k$) spaces.
The following extension and reduction operators, i.e., continuous cochain maps between the DS($k$) complex and the non-tensorized DDR($0$) complex, are built in \cite[Section~5.2]{Di-Pietro.Droniou.ea:26}:
\begin{equation}\label{eq:doublecomplex_ddr0_stokes}
  \begin{tikzcd}[column sep=3.5em]
    \text{DS($k$):}\hspace*{-3em}
    & 0 \arrow{r}{\ISgrad} & \XSgrad{\Th}\arrow{r}{\SGRAD{h}} \arrow[d, bend left, "\Rgradh"] & \XSrot{\Th}\arrow{r}{\SROT{h}}\arrow[d, bend left, "\Rroth"]&\Poly{k}(\Th)\arrow{r}{} \arrow[d, bend left, "\lproj{0}{h}"] & 0\\
    \text{DDR(0):}\hspace*{-3em}
    & 0 \arrow{r}{\Igrad{0}{h}}  & \Xgrad{0}{h}\arrow{r}{\uGh{0}} \arrow[u, bend left, "\Egradh"] & \Xcurl{0}{h}\arrow[u, bend left, "\Eroth"]\arrow{r}{\uCh{0}}&\Poly{0}(\Th)\arrow{r}{} \arrow[u, bend left, "i"]& 0.
  \end{tikzcd}
\end{equation}

The main property that we will use is that both $\Rgradh$ and $\Egradh$ leave the vertex (non-gradient) degrees of freedom unchanged, which implies
\begin{equation}\label{eq:ER-Id.V}
  (\Egradh\Rgradh\dof{z}{h}-\dof{z}{h})_V=0\qquad\forall V\in\Vh,\quad\forall \dof{z}{h}\in \XSgrad{\Th}.
\end{equation}

\begin{proposition}[Poincaré inequality on $\Image(\Egradh\Rgradh-\Id)$]
  \label{prop:pc.grad.coho}
  For all $\dof{q}{h}\in\Image(\Egradh\Rgradh-\Id)$, it holds
  \begin{equation}
    \label{eq:local.poincare.on.XSgrad}
    \nSgrad{h}{\dof{q}{h}} \lesssim h\nSrot{h}{\SGRAD{h}\dof{q}{h}}.
  \end{equation}
\end{proposition}

\begin{proof}
  The bound \eqref{eq:local.poincare.on.XSgrad} trivially follows if we establish its local version:
  \begin{equation}
    \label{eq:local.poincare.on.XSgrad.T}
    \nSgrad{T}{\dof{q}{T}} \lesssim h_T\nSrot{T}{\SGRAD{T}\dof{q}{T}}\qquad\forall T\in\Th.
  \end{equation}
  Let $T\in\Th$ and let us bound with the right-hand side of \eqref{eq:local.poincare.on.XSgrad.T} each term appearing in the definition \eqref{eq:def.norm.gradT} of $\nSgrad{T}{\dof{q}{T}}$.
  \medskip\\
  i) \emph{Vertex components}. By \eqref{eq:ER-Id.V}, we have $q_V=0$ for all $V\in\Vh$. The bound on the gradient components at the vertices is a consequence of the definition of $\SGRAD{T}$:
  \begin{equation}
    \label{eq:pc:Q_q,v}
    h_T^4\sum_{V\in\VT}|\Gqv|^2
    \overset{\eqref{eq:def.nablah}}=
    h_T^2\sum_{V\in\VT}h_T^2|(\SGRAD{h}\dof{q}{h})_V|^2
    \overset{\eqref{eq:def.norm.rotT}}\le
    h_T^2\nSrot{T}{\SGRAD{T}\dof{q}{T}}^2.
  \end{equation}
  \medskip\\
  ii) \emph{Edge components}. For all $E\in\ET$ and all $r\in\Poly{k}(E)$, by the definition \eqref{eq:def.Gqet} of $\Gqet$ it holds
  \begin{equation}\label{eq:q.r.G}
    \int_E q_E\, \partial_{\tangent_E} r \overset{\text{ $q_V=0$}}= -\int_E (\Gqet\dof{q}{E})\,r
    \overset{\eqref{eq:def.nablah}}= -\int_E (\SGRAD{T}\dofT{q})_E\cdot\tangent_E\,r
  \end{equation}
  Take $r\in\Poly{k}_0(E)$ such that $\partial_{\tangent_E}  r =q_E$; a discrete local Poincaré inequality \cite[Remark 1.46]{Di-Pietro.Droniou:20} yields $\norm{E}{r}\lesssim h_E\norm{E}{q_E}$. Using a Cauchy--Schwarz inequality on the right-hand side of \eqref{eq:q.r.G}, simplifying, raising the inequality to the square, multiplying by $h_T$ and summing over $E\in\ET$ gives
  \begin{align}
    \label{eq:pc:q_E}
    \sum_{E\in\ET} h_T \norm{E}{q_E}^2 \lesssim \sum_{E\in\ET} h_T^3\norm{E}{(\SGRAD{T}\dofT{q})_E}^2 \overset{{\eqref{eq:def.norm.rotT}}}{\le}h_T^2\nSrot{T}{\SGRAD{T}\dof{q}{T}}^2.
  \end{align}
  For all $E\in\ET$, the bound on $\Gqen$ is a straightforward consequence of the definitions:
  \begin{align}
    \label{eq:pc:G_q,E}
    \sum_{E\in\ET} h_T^3 \norm{E}{\Gqen}^2 \overset{\eqref{eq:def.nablah}}= \sum_{E\in\ET} h_T^3 \norm{E}{(\SGRAD{T}\dofT{q})_E\cdot \normal_{E}}^2 \overset{\eqref{eq:def.norm.rotT}}{\le} h_T^2\nSrot{T}{\SGRAD{T}\dof{q}{T}}^2.
  \end{align}
  \\
  iii) \emph{Element components}. The definition \eqref{def:nablaT} of $\nablaT$ gives, for all $\bvec w\in\cRoly{k-1}(T)\subset \Poly{k-1}(T)^2$,
  \begin{equation}\label{eq:pc.ER.grad}
    \int_T q_T \DIV \bvec{w}
    = - \int_T \nablaT \dofT{q}\cdot \bvec{w}
    + \sum_{E\in\ET}\omega_{TE}\int_{E} q_E\,(\bvec{w}\cdot\normal_{E}).
  \end{equation}
  Since $\DIV : \cRoly{k-1}(T) \rightarrow \Poly{k-2}(T)$ is an isomorphism, we can take $\bvec w\in\cRoly{k-1}(T)$ such that $\DIV\bvec  w = q_T$ and $\norm{T}{\bvec{w}}\lesssim h_T\norm{T}{q_T}$ by \cite[Lemma 9]{Di-Pietro.Droniou:23}. We then plug this $\bvec w$ into \eqref{eq:pc.ER.grad}, use Cauchy--Schwarz inequalities together with a discrete trace inequality on $\bvec w$, simplify by $\norm{T}{q_T}$, and square to obtain
  \begin{align}
    \label{eq:pc:q_T}
    \norm{T}{q_T}^2 \lesssim  h_T^2\norm{T}{\nablaT \dofT{q}}^2 +  \sum_{E\in\ET}h_T\norm{E}{q_E}^2 \overset{{\eqref{eq:def.norm.rotT}},\eqref{eq:pc:q_E}}{\lesssim}h_T^2 \nSrot{T}{\SGRAD{T}\dof{q}{T}}^2.
  \end{align}
  Finally, summing \eqref{eq:pc:Q_q,v}, \eqref{eq:pc:q_E}, \eqref{eq:pc:G_q,E}, and \eqref{eq:pc:q_T}, then taking the square root of the resulting inequality, we obtain \eqref{eq:local.poincare.on.XSgrad.T}.
\end{proof}

\begin{proposition}[Poincaré inequality on $\Image(\Eroth\Rroth-\Id)$]
  \label{prop:pc.curl.coho}
  For all $\dof{\bvec{v}}{h}\in\XSrot{\Th}$, there exists $\dofh{\bvec{z}}\in\XSrot{\Th}$ such that
  \begin{equation}
    \label{eq:pc.curl.coho}
    ( \lproj{0}{h}-\Id)\SROT{h} \dofh{\bvec{v}} = \SROT{h} \dofh{\bvec{z}} \quad\text{ and }\quad\nSrot{h}{\dofh{\bvec{z}}}\lesssim  h\norm{\Omega}{\SROT{h}\dof{\bvec{z}}{h}}.
  \end{equation}
\end{proposition}

\begin{proof}
  We define $\dofh{\bvec z}\in\XSrot{\Th}$ component by component. We set, for all $V\in\Vh$, $\bvec z_V=\bvec 0$ and, for all $E\in\Eh$, $\bvec{z}_E=\bvec 0$. For all $T\in\Th$, $\bvec{z}_T\in\Roly{k-1}(T)$ is selected such that
  \begin{equation}
    \label{eq:proof.pc.curl.const.z.T}
    \int_T \bvec{z}_T\cdot\CURL r = \int_T [(\lproj{0}{T}-\Id)\SROT{T} \dof{\bvec{v}}{T}] \,r \qquad \forall r\in\Poly{k}_0(T).
  \end{equation}
  This relation also holds for $r$ constant, by definition of $\lproj{0}{T}$. Since the edge components of $\dof{z}{h}$ vanish, combining \eqref{eq:proof.pc.curl.const.z.T} (for all $r\in\Poly{k}(T)$) with the definition \eqref{eq:def.SROT} of $\SROT{T}\dof{\bvec{z}}{T}$ shows that $\SROT{T}\dofT{\bvec z}=(\lproj{0}{T}-\Id)\SROT{T} \dof{\bvec{v}}{T}$.

  It remains to establish the estimate in \eqref{eq:pc.curl.coho}. Since the vertex and edge components of $\dofh{\bvec z}$ vanish, only the element components remain to be bounded.
  Recalling that $\bvec{z}_T\in\Roly{k-1}(T)$, we can take $r\in\Poly{k}_0(T)$ such that $\CURL r = \bvec{\bvec z}_T$ in \eqref{eq:proof.pc.curl.const.z.T}, and the local Poincaré inequality \cite[Remark 1.46]{Di-Pietro.Droniou:20} gives $\norm{T}{r}\lesssim h_T\norm{T}{\bvec{z}_T}$. Since $\SROT{T}\dof{\bvec z}{T}=(\lproj{0}{T}-\Id)\SROT{T} \dof{\bvec{v}}{T}$, a Cauchy--Schwarz inequality then yields $\norm{T}{\bvec{z}_T}^2 \lesssim h_T\norm{T}{\SROT{T}\dof{\bvec z}{T}}\norm{T}{\bvec{z}_T}$. Simplifying, squaring, summing over $T\in\Th$ and using $h_T\le h$ concludes the proof.
\end{proof}

\subsubsection{Proof of Theorems \ref{th:pc.grad} and \ref{th:pc.rot}}\label{sec:proof.poincares}

\begin{proof}[Proof of Theorem \ref{th:pc.grad}]
  The result is a direct consequence of Propositions \ref{prop:pc.from.reduced.spaces} in the appendix and \ref{prop:pc.grad.coho} along with the Poincaré inequality for the discrete gradient in the DDR(0) sequence, see \cite[Theorem 3]{Di-Pietro.Droniou:23} in the 3D case and \cite[Corollary 5]{Pietro.Droniou.ea:25} in any dimension. Indeed, Proposition \ref{prop:pc.grad.coho} implies Assumption \ref{ass:pc.reduced} on $\XSgrad{\Th}$: for any $\dofh{x}\in\XSgrad{\Th}$, simply set $\dofh{z} \coloneqq \Egradh\Rgradh\dofh{x}-\dofh{x}$ and apply Proposition \ref{prop:pc.grad.coho} with $\dofh{q}=\dofh{z}$. Furthermore, by boundedness of the $L^2$-orthogonal projectors, Cauchy--Schwarz inequalities and the local Poincaré inequalities of \cite[Lemma 9]{Di-Pietro.Droniou:23}, one can easily check from their definition in \cite[Section 5.2]{Di-Pietro.Droniou.ea:26} that the extension and reduction maps in \eqref{eq:doublecomplex_ddr0_stokes} are continuous uniformly in $h$, which ensures that the constant in the right-hand side of \eqref{eq:poincare.by.transfer} remains bounded uniformly in $h$.
\end{proof}

\begin{proof}[Proof of Theorem \ref{th:pc.rot}]
  The result is a direct consequence of Proposition \ref{prop:pc.from.reduced.spaces} in the appendix, together with Proposition \ref{prop:pc.curl.coho} and the Poincaré inequality for the discrete rotor in DDR(0) (2D version of \cite[Theorem 4]{Di-Pietro.Droniou:23}, see \cite[Corollary 5]{Pietro.Droniou.ea:25} for a proof in any dimension). Proposition \ref{prop:pc.curl.coho} implies Assumption \ref{ass:pc.reduced} thanks to the cochain maps property of the extension and reductions. The conclusion follows, as in the proof of Theorem \ref{th:pc.grad}, by noticing that the extension and reduction are continuous uniformly in $h$.
\end{proof}

\subsubsection{Poincaré inequalities}\label{sec:pc.dirichlet}

We briefly discuss here the case of Poincaré inequalities for the discrete version of the complex \eqref{eq:bgg.complex.dirichlet} with Dirichlet boundary conditions.
Define the spaces  $\XSgrado{\Th}$ and $\tXdRgrado{\Th}$ as the discrete counterparts of $H_{2,0}(\Omega)$ and $H_{1,0}(\Omega)^2$, respectively:
\begin{equation}
  \label{eq:def.XSgrado}
  \XSgrado{\Th}\coloneq \left\{
  \dofh{q}\in \XSgrad{\Th} \,:\,  \ q_E = G_{q,E}^n = 0 \quad \forall  E \in \Eh^{\rm b},\,%
  \text{$q_V=0$ and $\bvec{G}_{q,V}=\bvec{0}$} \quad \forall V \in \Vh^{\rm b} \right\},
\end{equation}
and
\begin{equation}\label{eq:def.txdrgrado}
  \tXdRgrado{\Th}\coloneq \left\{
  \dofh{\bvec v} \in \tXdRgrad{\Th}
  \,:\, \bvec{v}_E= \bvec 0 \quad \forall E \in\Eh^{\rm b},\,%
  \bvec{v}_V = \bvec 0 \quad \forall V \in \Vh^{\rm b}
  \right\}.
\end{equation}
where $\Eh^{\rm b} \subset \Eh$ and $\Vh^{\rm b} \subset \Vh$ respectively gather the edges and vertices contained in $\partial\Omega$.

\begin{theorem}[Poincaré inequalities on $\XSgrado{\Th}$ and $\tXdRgrado{\Th}$]
  \label{th:pc.adapted}
  The following Poincaré inequalities hold:
  \begin{alignat*}{2}
    \nSgrad{h}{\dof{q}{h}} &\lesssim \nSrot{h}{\SGRAD{h}\dof{q}{h}}\qquad
    &&\forall \dofh{q}\in\XSgrado{\Th},\\
    \nSrot{h}{\dof{\bvec{v}}{h}} &\lesssim \norm{\Omega}{\SROT{h}\dof{\bvec{v}}{h}} \qquad &&\forall \dof{\bvec{v}}{h}\in\tXdRgrado{\Th}\cap(\Ker\SROT{h})^{\perp}.
  \end{alignat*}
\end{theorem}

\begin{proof}
  We again apply the abstract framework of Section~\ref{appendix:abstract.poincare}, but this time with the first two spaces in the first row of the diagram \eqref{eq:doublecomplex_ddr0_stokes} replaced by $\XSgrado{\Th}$ and $\tXdRgrado{\Th}$ (top row), and the first two spaces in the second row by the DDR(0) spaces with zero boundary conditions defined as follows:
  \[
  \begin{aligned}
    \underline{H}_{1,0}^{0}(\Th) &\coloneq \left\{(q_V)_{V\in\Vh}\,:\, q_V\in\mathbb{R}\quad \forall V\in \Vh \setminus \Vh^{\rm b},\; q_V = 0 \quad \forall V\in\Vh^b \right\}, \\
    \underline{\bvec{H}}_{\ROT,0}^0(\Th) &\coloneq \left\{(v_E)_{E\in\Eh}\,:\, v_E\in\mathbb{R}\quad \forall E\in \Eh \setminus \Eh^{\rm b},\; v_E = 0 \quad \forall E\in\Eh^b \right\}.
  \end{aligned}
  \]
  It is easily checked that the extension and reduction operators still correctly map these spaces, and that Propositions \ref{prop:pc.grad.coho} and \ref{prop:pc.curl.coho} hold with these boundary conditions.
  Since Poincaré inequalities have been established for the DDR(0) complex with homogeneous boundary conditions (see \cite[Lemmas 7 and 8]{Di-Pietro:24}), the abstract framework of Appendix \ref{appendix:abstract.poincare} applies and yields the Poincaré inequalities stated in the theorem, after noticing that $\Ker((\SGRAD{h})_{|\XSgrado{\Th}})=\{0\}$ (since $\Ker(\SGRAD{h})$ is made of interpolates of piecewise constant functions on the connected components of $\Omega$ by \cite[Theorem 10]{Di-Pietro.Droniou.ea:26}).
\end{proof}

\section{Application to the Kirchhoff--Love plate model}\label{sec:kirchhoff}

\subsection{Discrete problem}

To design a numerical scheme for problem \eqref{eq:KL:weak:essential}, we consider the Hodge Laplacian associated with the discrete Hessian complex \eqref{derived-BGG}. Boundary conditions are incorporated by replacing $\XSgrad{\Th}$ with $\XSgrado{\Th}$, defined in \eqref{eq:def.XSgrado}, and $\XdRrotS{\Th}$ by $\XdRrotSo{\Th}$ defined as follows:
\[
\XdRrotSo{\Th} \coloneq  \left\{ \dofh{\bvec \xi} \in \XdRrotS{\Th}\,:\, \bvec\xi_{E} = \bvec 0 \quad \forall E \in\Eh^{\rm b}\right\}.
\]
We obtain the discrete problem:
Find $\dofh{u}\in\XSgrado{\Th}$ such that
\begin{equation}\label{eq:discrete.kl}
  (\Hess{h} \dofh{u}, \Hess{h} \dofh{q} )_{\VROT,h} = ( f, \SPgradh\dofh{q} )_{L^2(\Omega)}\qquad \forall \dofh{q}\in\XSgrado{\Th}.
\end{equation}
The discrete scalar product $\lparen \cdot ,\cdot \rparen_{\VROT,h}$ is the tensorized version of the one defined on $\XdRrot{\Th}$ by \cite[Eq.~(4.15)]{Di-Pietro.Droniou:23}:
For all $(\uvec{\upsilon}_T, \uvec{\tau}_T) \in \tXdRrot{T} \times \tXdRrot{T}$,
\[
\begin{gathered}
  (\uvec{\upsilon}_T, \uvec{\tau}_T)_{\VROT,T}
  \coloneqq
  \int_T \Prot{k+1}{T} \uvec{\upsilon}_T : \Prot{k+1}{T} \uvec{\tau}_T
  + s_{\VROT,T}(\uvec{\upsilon}_T, \uvec{\tau}_T),
  \\
  s_{\VROT,T}(\uvec{\upsilon}_T, \uvec{\tau}_T)
  \coloneqq
  \sum_{E\in\ET} h_E \int_E \left(
  (\Prot{k+1}{T} \uvec{\upsilon}_T) \tangent_E - \bvec{\upsilon}_E
  \right) \cdot \left(
  (\Prot{k+1}{T} \uvec{\tau}_T) \tangent_E - \bvec{\tau}_E
  \right),
\end{gathered}
\]
where $\Prot{k+1}{T} : \XdRrot{T} \to \Poly{k}(T)^{d \times d}$  is the local potential reconstruction on $\XdRrot{T}$.

\subsection{Analytical properties of the discrete Hessian complex}

\begin{theorem}[Poincaré inequalities on $\XSgrado{\Th}$]
  \label{th:pc.hess}
  The following discrete Poincaré inequalities hold:
  For all $\dof{q}{h}\in\XSgrado{\Th}$,
  \[
  \norm{2,h}{\dof{q}{h}} \lesssim \norm{\VROT,h}{\Hess{h}\dof{q}{h}},\qquad
  \norm{1,h}{\SGRAD{h}\dofh{q}}\lesssim\norm{\VROT,h}{\Hess{h}\dofh{q}}.
  \]
\end{theorem}

\begin{remark}[$\XSgrad{\Th}$]\label{eq:norm.VROT.h.Hess}
  The above Poincaré inequalities show that $\norm{\VROT,h}{\Hess{h}\cdot}$ is a norm on $\XSgrad{\Th}$.
\end{remark}

\begin{proof}
  Notice first that the proof of the Poincaré--Wirtinger inequality for $\tGRAD{h}$ stated in \cite[Theorem 3]{Di-Pietro.Droniou:23} can be easily adapted into a Poincaré inequality on $\tXdRgrado{\Th}$ (see \eqref{eq:def.txdrgrado}).

  Moreover, for all $\dofh{q}\in\XSgrado{\Th}$ and all $E \in \Eh^{\rm b}$, since all the polynomials in $\dof{q}{E}$ vanish, we have $\Gqet \dof{q}{E}=0$. By the definition \eqref{eq:def.nablah} of $\SGRAD{h}$, it follows that $\SGRAD{h}\dofh{q} \in \tXdRgrado{\Th}$. Therefore, for all $\dofh{q}\in\XSgrado{\Th}$, the Poincaré inequalities on $\XSgrado{\Th}$ and $\tXdRgrado{\Th}$ yield
  \[
  \norm{2,h}{\dofh{q}}
  \lesssim
  \norm{1,h}{\SGRAD{h}\dofh{q}}
  \lesssim
  \norm{\VROT,h}{\tGRAD{h}\SGRAD{h}\dofh{q}}
  =
  \norm{\VROT,h}{\Hess{h}\dofh{q}}.\qedhere
  \]
\end{proof}

In order to derive an error estimate for the scheme \eqref{eq:discrete.kl}, we first state an adjoint consistency result.
Hereafter, we set
\[
\bvec{H}_{\DIV\VDIV,0}(\Omega)\coloneq
\begin{aligned}[t]
  \Big\{{}&\bvec{\xi}\in L^2(\Omega)^{2\times2}\,:\,
  \VDIV\bvec{\xi}\in L^2(\Omega)^2\,,\;\DIV\VDIV\bvec{\xi}\in L^2(\Omega)\,,\\
  &\bvec{\xi}\normal_{\Omega}=\bvec{0}\text{ and }(\VDIV\bvec{\xi})\cdot\normal_\Omega=0\text{ on $\partial\Omega$}\Big\}
\end{aligned}
\]

\begin{theorem}[Adjoint consistency of $\Hess{h}$]
  For all $\bvec \xi \in \Cspace{0}(\overline{\Omega})^{2\times 2}\cap H_{\DIV\VDIV,0}(\Omega)$ such that $\bvec \xi \in H_{k+2}(\Th)^{2 \times 2} $ and $\VDIV\bvec \xi\in H_{k+3}(\Th)^2$, and all $\dofh{q}\in\XSgrad{\Th}$, it holds
  \begin{multline}\label{eq:adj.const.hess}
    \Bigg|\sum_{T\in\Th}
    (\IdRrotT \bvec \xi, \Hess{T}\dofT{q})_{\VROT,T} - \int_T \DIV\VDIV\bvec\xi \, \SPgrad\dofT{q}
    \Bigg|
    \lesssim  h^{k+2} \seminorm{H_{k+2}(\Th)^{2\times 2}}{\bvec \xi}  \norm{\VROT,h}{\Hess{h}\dofh{q}}
    \\
    + \left(
    h^{k+3} \seminorm{H_{k+3}(\Th)^2}{\VDIV \bvec \xi}
    + h^{k+1}  \seminorm{H_{(k+1,2)}(\Th)^2}{\VDIV\bvec\xi}
    \right) \norm{1,h}{\SGRAD{h}\dofh{q}}.
  \end{multline}
\end{theorem}

\begin{proof}
  For all  $\dofh{q}\in\XSgrad{\Th}$, we recall that $\Hess{T}\coloneq\tGRAD{T}\circ \SGRAD{T}$,
  add and remove $\int_T \VDIV\xi \cdot \Pgrad{k+2}{T}\SGRAD{T} \dofT{q} + (\IdRgradT \VDIV \bvec \xi, \SGRAD{T}\dofT{q})_{1,T}$ to write
  \[
  \begin{aligned}
    &\sum_{T\in\Th}  (\IdRrotT \bvec \xi, \Hess{T}\dofT{q})_{\VROT,T} - \int_T \DIV\VDIV\bvec\xi \, \SPgrad\dofT{q} 
    \\
    &= \sum_{T\in\Th} 
    (\IdRrot \bvec \xi, \tGRAD{T}(\SGRAD{T}\dofT{q}))_{\VROT,T} + \int_T \VDIV\xi \cdot \Pgrad{k+2}{T}\SGRAD{T} \dofT{q}
    \\
    &\quad+ \sum_{T\in\Th} 
    -\int_T \VDIV\xi \cdot \Pgrad{k+2}{T}\SGRAD{T} \dofT{q} + (\IdRgradT\VDIV \xi, \SGRAD{T}\dofT{q})_{1,T}
    \\
    &\quad + \sum_{T\in\Th} 
    -  (\IdRgradT \VDIV \bvec \xi, \SGRAD{T}\dofT{q})_{1,T} - \int_T \DIV\VDIV\bvec\xi \, \SPgrad\dofT{q}.
  \end{aligned}
  \]
  The conclusion follows using the adjoint consistency \cite[Theorem 9]{Di-Pietro.Droniou:23} of $\tGRAD{h}$ to estimate the first summation,
  the consistency \cite[Eq.~(6.11)]{Di-Pietro.Droniou:23} of the scalar product $(\cdot,\cdot)_{1,T}$ to estimate the second summation,
  and the adjoint consistency \eqref{eq:adjoint.const.Sgrad} of $\SGRAD{h}$ to estimate the last summation.
\end{proof}

\begin{theorem}[Error estimate] \label{th:error.estimate}
  Let $u\in H_{2,0}(\Omega)$ be the solution to problem \eqref{eq:KL:weak:essential}, and let $\dofh{u}\in\XSgrado{\Th}$ denote the solution to problem \eqref{eq:discrete.kl}.  Assume the additional regularity $u\in \Cspace{0}(\overline{\Omega})\cap H_{k+4}(\Th)$ and $\Delta u \in H_{k+4}(\Th)$. Then, the following estimate holds:
  \begin{equation*}
    \norm{\VROT,h}{\Hess{h}(\dofh{u} - \ISgrad u)} \lesssim (h^{k+2} \seminorm{H_{k+4}(\Th)}{u}+h^{k+3} \seminorm{H_{k+4}(\Th)}{\Delta u} + h^{k+1}   \seminorm{H_{(k+1,2)}(\Th)^2}{\GRAD \Delta u}).
  \end{equation*}
\end{theorem}

\begin{proof}
  Endowing $\XSgrado{\Th}$ with the norm $\norm{\VROT,h}{\Hess{h}\cdot}$ (see Remark \ref{eq:norm.VROT.h.Hess}), the bilinear form in the scheme  \eqref{eq:discrete.kl} is obviously coercive (with constant $1$).
  We can therefore apply the third Strang Lemma \cite[Theorem 10]{Di-Pietro.Droniou:18} with
  $a_h(\cdot,\cdot) \coloneq (\cdot,\cdot)_{\VROT,h}$ and
  $l_h(\cdot) \coloneq ( f, \SPgradh \cdot )_{L^2(\Omega)}$
  to obtain the initial estimate
  \begin{equation}\label{eq:kl:basic.estimate}
    \norm{\VROT,h}{\Hess{h}(\dofh{u} - \ISgrad u)} \lesssim
    \sup_{\dofh{q}\in\XSgrado{\Th}\setminus \left\{ 0 \right\}}\frac{\mathcal{E}(u,\dofh{q})
    }{\norm{\VROT,h}{\Hess{h}\dofh{q}}},
  \end{equation}
  where the consistency error $\mathcal{E}(u,\cdot): \XSgrado{\Th} \to \Real$ is defined by
  \[
  \mathcal{E}(u, \dofh{q})\coloneq
  ( f, \SPgradh \dofh{q})_{L^2(\Omega)}-( \Hess{h} \ISgrad u, \Hess{h} \dofh{q})_{\VROT,h}.
  \]
  It remains to estimate the right-hand side of \eqref{eq:kl:basic.estimate}.
  For all $\dofh{q}\in\XSgrad{\Th}$, using $f=\Delta^2 u=\DIV\VDIV\hess u$ together with the commutation property \eqref{eq:com.prop.hess}, we obtain
  \[
  \begin{aligned}
    &\mathcal{E}(u,\dofh{q}) = \sum_{T\in\Th} \left[ ( \DIV \VDIV \hess u, \SPgrad \dofT{q} )_{L^2(T)} -( \IdRrotT \hess u , \Hess{T} \dofT{q})_{\VROT,T}\right]
    \\
    &\quad\lesssim \left(
    h^{k+2} \seminorm{H_{k+2}(\Th)^{2\times 2}}{\hess u}
    +h^{k+3} \seminorm{H_{k+3}(\Th)^2}{\VDIV \hess u}
    + h^{k+1}  \seminorm{H_{(k+1,2)}(\Th)^2}{\VDIV\hess u}
    \right)
    \norm{\VROT,h}{\Hess{h}\dofh{q}}\\
    &\quad\lesssim \left(
    h^{k+2} \seminorm{H_{k+2}(\Th)^{2\times 2}}{\hess u}
    +h^{k+3} \seminorm{H_{k+4}(\Th)}{\Delta u}
    + h^{k+1}  \seminorm{H_{(k+1,2)}(\Th)^2}{\GRAD \Delta u}
    \right)
    \norm{\VROT,h}{\Hess{h}\dofh{q}},
  \end{aligned}
  \]
  where the estimate follows from the adjoint consistency result \eqref{eq:adj.const.hess} with $\bvec{\xi} = -\hess u$ together with the Poincaré inequality stated in Theorem \ref{th:pc.hess} and the identity $\VDIV \hess u = \GRAD \Delta u$.
  Plugging this estimate into \eqref{eq:kl:basic.estimate} and simplifying concludes the proof. 
\end{proof}

\subsection{Numerical results}

This section presents numerical results that demonstrate the correct behavior of the scheme \eqref{eq:discrete.kl}, in agreement with the convergence result established in Theorem \ref{th:error.estimate}. The computational domain is $\Omega = (0,1)^2$, and we consider meshes that are Cartesian, hexagonal and triangular, as shown in Figure \ref{fig:meshes}.

\begin{figure}\centering
  \begin{minipage}{0.3\textwidth}\centering
    \includegraphics[height=4cm]{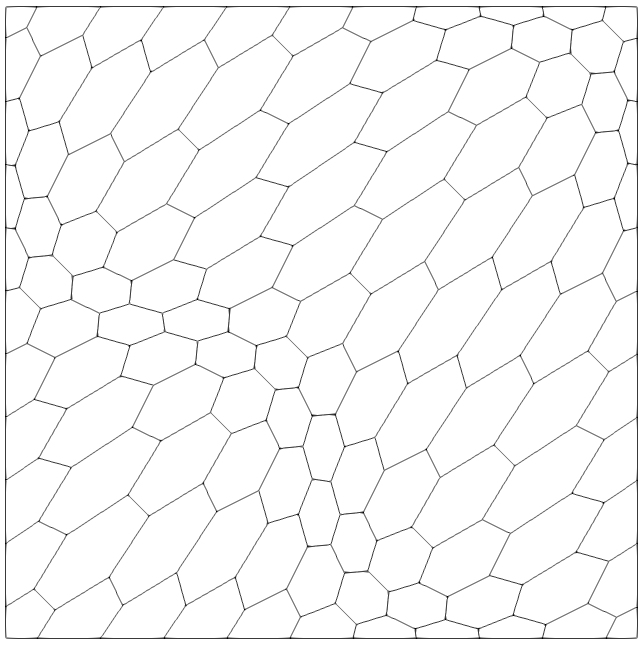}
    \subcaption{Hexagonal mesh\label{fig:hexa}}
  \end{minipage}
  \hspace{0.25cm}
  \begin{minipage}{0.3\textwidth}\centering
    \includegraphics[height=4cm]{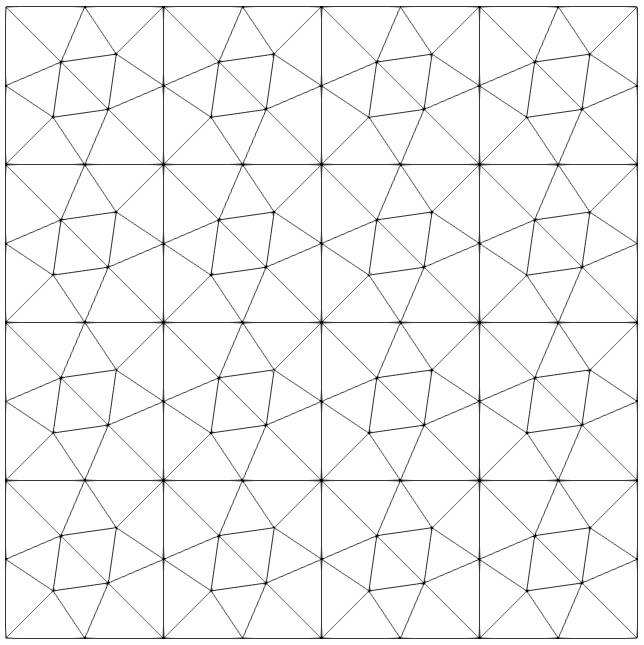}
    \subcaption{Triangular mesh}
  \end{minipage}
  \vspace{0.25cm}
  \begin{minipage}{0.3\textwidth}\centering
    \includegraphics[height=4cm]{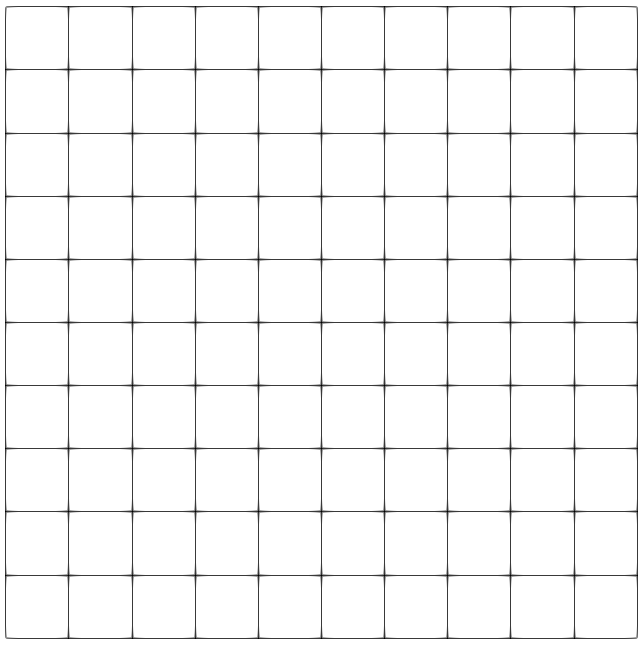}
    \subcaption{Cartesian mesh}
  \end{minipage}
  \caption{Meshes from the families used in the numerical tests.\label{fig:meshes}}
\end{figure}

We conduct two series of tests. In the first, we consider a problem with manufactured right-hand side yielding the trigonometric solution given by $u(\mathbf{x}) = \sin(\pi x_1)\sin(\pi x_2)$.
We define the relative energy error $E_h$ as
\begin{equation}\label{eq:def.Eh}
  E_h\coloneq \frac{\norm{\VROT,h}{\Hess{h}(\dofh{u} - \ISgrad u)}}{\norm{\VROT,h}{\Hess{h} \ISgrad u}},
\end{equation}
where $u$ and $\dofh{u}$ are the exact and the approximate solution, respectively.
The results are presented in Figure \ref{fig:trigo.sol} for different degrees $k\in \llbracket 0, 3\rrbracket$, and show a convergence of the energy error with rate $h^{k+1}$ as expected.

\begin{figure}\centering
  \ref{conv.hexa}
  \vspace{0.50cm}\\
  \begin{minipage}{0.45\textwidth}
    \begin{tikzpicture}[scale=0.85]
      \begin{loglogaxis}
        \addplot [thick, mark=star, red] table[x=MeshSize,y=EnError] {outputs/trigo_solution/hexa_k0/data_rates.dat};
        \logLogSlopeTriangle{0.90}{0.4}{0.1}{1}{black};
        \addplot [thick, mark=*, blue] table[x=MeshSize,y=EnError] {outputs/trigo_solution/hexa_k1/data_rates.dat};
        \logLogSlopeTriangle{0.90}{0.4}{0.1}{2}{black};
        \addplot [thick, mark=square*, black] table[x=MeshSize,y=EnError] {outputs/trigo_solution/hexa_k2/data_rates.dat};
        \logLogSlopeTriangle{0.90}{0.4}{0.1}{3}{black};
      \end{loglogaxis}
    \end{tikzpicture}
    \subcaption{Hexagonal mesh}
  \end{minipage}
  \begin{minipage}{0.45\textwidth}
    \begin{tikzpicture}[scale=0.85]
      \begin{loglogaxis}
        \addplot [thick, mark=star, red] table[x=MeshSize,y=EnError] {outputs/trigo_solution/tri_k0/data_rates.dat};
        \logLogSlopeTriangle{0.90}{0.4}{0.1}{1}{black};
        \addplot [thick, mark=*, blue] table[x=MeshSize,y=EnError] {outputs/trigo_solution/tri_k1/data_rates.dat};
        \logLogSlopeTriangle{0.90}{0.4}{0.1}{2}{black};
        \addplot [thick, mark=square*, black] table[x=MeshSize,y=EnError] {outputs/trigo_solution/tri_k2/data_rates.dat};
        \logLogSlopeTriangle{0.90}{0.4}{0.1}{3}{black};
      \end{loglogaxis}
    \end{tikzpicture}
    \subcaption{Triangular mesh}\label{fig:trigo.tri}
  \end{minipage}
  \vspace{0.25cm}\\
  \begin{minipage}{0.45\textwidth}
    \begin{tikzpicture}[scale=0.85]
      \begin{loglogaxis}[legend columns=3, legend to name=conv.hexa]
        \addplot [thick, mark=star, red] table[x=MeshSize,y=EnError] {outputs/trigo_solution/cart_k0/data_rates.dat};
        \addlegendentry{$k=0$}
        \logLogSlopeTriangle{0.90}{0.4}{0.1}{1}{black};
        \addplot [thick, mark=*, blue] table[x=MeshSize,y=EnError] {outputs/trigo_solution/cart_k1/data_rates.dat};
        \addlegendentry{$k=1$}
        \logLogSlopeTriangle{0.90}{0.4}{0.1}{2}{black};
        \addplot [thick, mark=square*, black] table[x=MeshSize,y=EnError] {outputs/trigo_solution/cart_k2/data_rates.dat};
        \addlegendentry{$k=2$}
        \logLogSlopeTriangle{0.90}{0.4}{0.1}{3}{black};
      \end{loglogaxis}
    \end{tikzpicture}
    \subcaption{Cartesian mesh}
  \end{minipage}
  \caption{Error $E_h$ (see \eqref{eq:def.Eh}) with respect to the meshsize $h$ for the trigonometric solution. \label{fig:conv.trigo}}
  \label{fig:trigo.sol}
\end{figure}
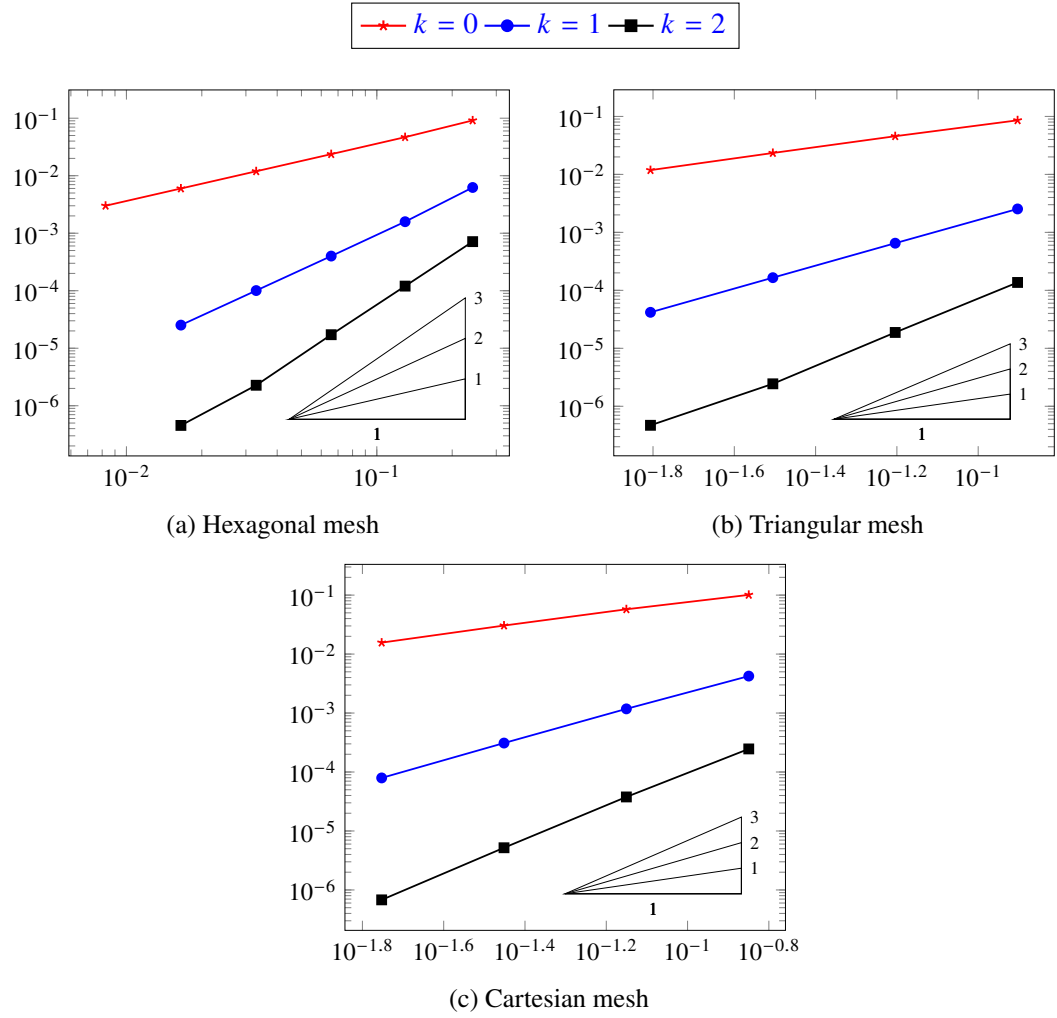

For the second test, a Dirac force is applied to the plate at location $(0.325, 0.732)$. Figure \ref{fig:visu} provides a visual representation of the approximate solution obtained on the finest triangular mesh (corresponding to $h=0.008$) and polynomial degree $k=2$. This reference solution yields an approximate energy value of $0.05513$.
The aim of this test is to illustrate one of the key advantages of polygonal schemes by comparing the convergence toward this reference energy using a family of meshes (see Figure \ref{fig:loc_ref}) that are locally refined around the Dirac point. Such locally refined meshes have ``hanging nodes'' that make them difficult to handle with standard finite element methods; on the contrary, schemes supporting generic polygonal elements can seamlessly handle such nodes, by splitting physical edges (the cells containing these nodes are no longer squares, but pentagons, hexagons, etc.). The corresponding results are reported in Figure \ref{fig:dirac.energy}, the error on the energy being defined as:
\begin{equation}\label{eq:def.Eh.dirac}
  \tilde{E}_h\coloneq \norm{\VROT,h}{\Hess{h}(\dofh{u})} - 0.05513,
\end{equation}
where $\dofh{u}$ is the approximate solution.
For a given target accuracy, the locally refined meshes require fewer degrees of freedom than the standard mesh families. Moreover, the reference value is reached faster with local refinement than other mesh families.
The left panel of Figure~\ref{fig:dirac.energy} shows that, for $k=0$, the hexagonal meshes are less efficient than the other mesh families, as a larger number of degrees of freedom is required to achieve a comparable accuracy. The same figure also illustrates one of the advantages of the present method, namely its arbitrary-order construction. Indeed, the right panel shows that, on hexagonal meshes, increasing the polynomial degree significantly improves the approximation quality for a given number of degrees of freedom. This higher-order approximation compensates for the lower efficiency observed at the lowest order and allows the reference value to be reached more rapidly.

\begin{figure}\centering
  \includegraphics[height=5cm]{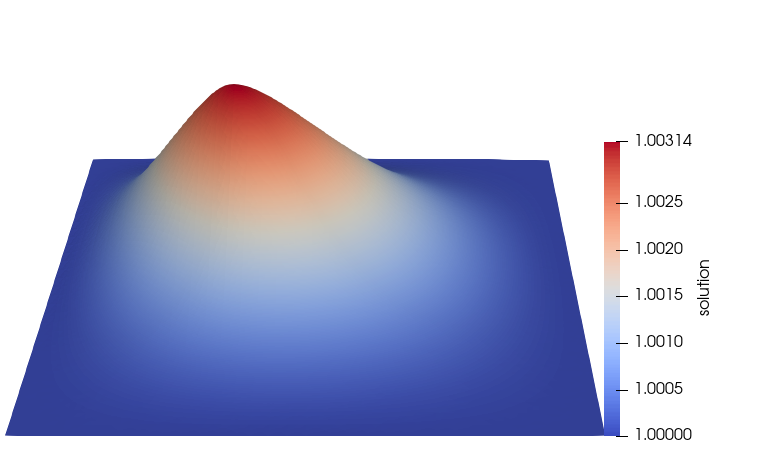}
  \caption{Approximate solution for a Dirac mass source term, located at $(0.325, 0.732)$.\label{fig:visu}}
\end{figure}

\begin{figure}\centering
  \begin{minipage}{0.3\textwidth}\centering
    \includegraphics[height=4cm]{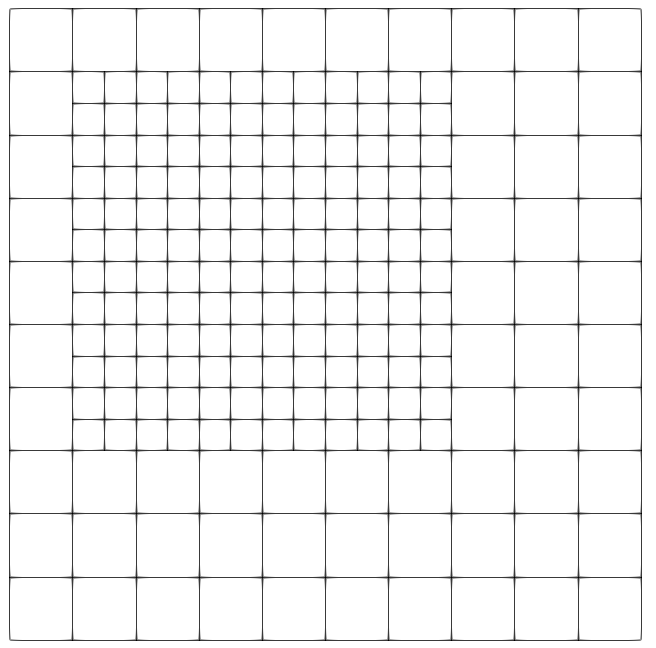}
  \end{minipage}
  \hspace{0.25cm}
  \begin{minipage}{0.3\textwidth}\centering
    \includegraphics[height=4cm]{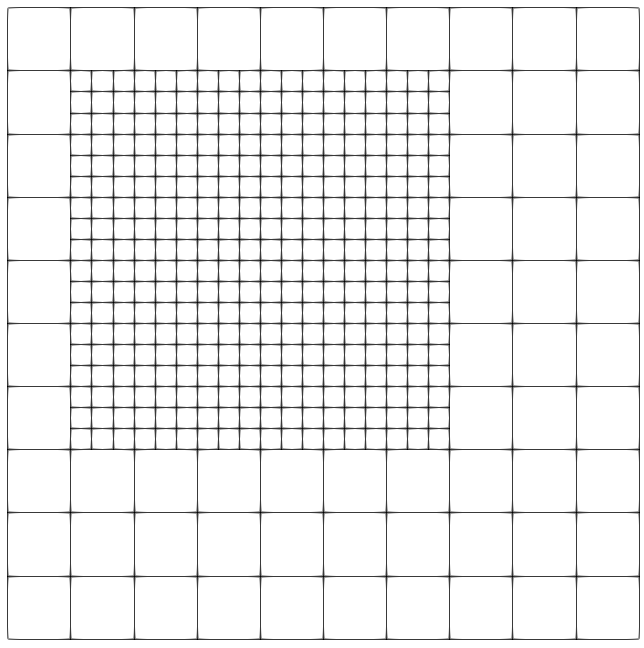}
  \end{minipage}
  \vspace{0.25cm}
  \begin{minipage}{0.3\textwidth}\centering
    \includegraphics[height=4cm]{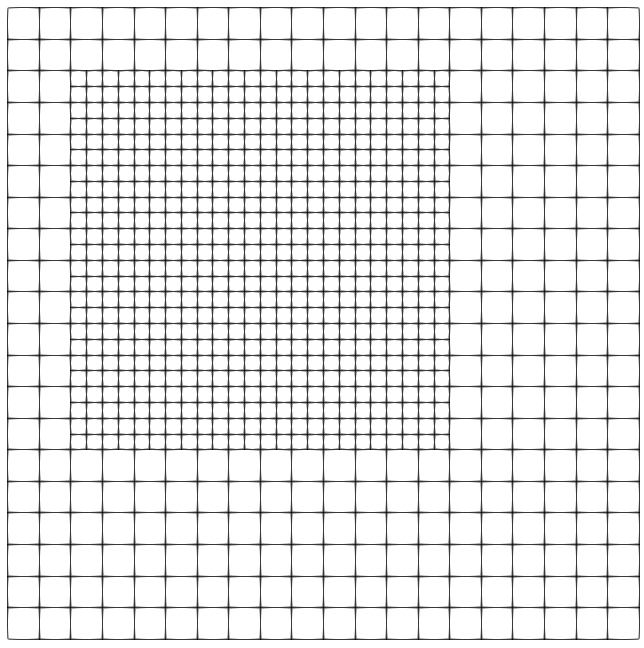}
  \end{minipage}
  \caption{Locally refined Cartesian meshes.}
  \label{fig:loc_ref}
\end{figure}

\begin{figure}\centering

  \begin{minipage}{0.48\textwidth}
    \centering
    \begin{tikzpicture}[scale=0.80]
      \begin{semilogxaxis}[
          xlabel={Number of DOFs in $\XSgrad{\Th}$},
          ylabel={$\tilde{E}_h$},
          legend pos=south east
        ]
        \addplot [thick, mark=star, red] table[x=DimXHess,y=EnApprox] {outputs/dirac_solution/cart_k0/data_rates.dat};
        \addlegendentry{cartesian mesh}
        \addplot [thick, mark=*, blue] table[x=DimXHess,y=EnApprox] {outputs/dirac_solution/cart_loc_ref_k0/data_rates.dat};
        \addlegendentry{locally refined mesh}
        \addplot [thick, mark=square*, black] table[x=DimXHess,y=EnApprox] {outputs/dirac_solution/tri_k0/data_rates.dat};
        \addlegendentry{triangular mesh}
        \addplot [thick, mark=o, olive] table[x=DimXHess,y=EnApprox] {outputs/dirac_solution/hexa_k0/data_rates.dat};
        \addlegendentry{hexagonal mesh}
      \end{semilogxaxis}
    \end{tikzpicture}

    \vspace{0.15cm}
    \centerline{(a) $k=0$, various mesh families.}
  \end{minipage}
  \hfill
  \begin{minipage}{0.48\textwidth}
    \centering
    \begin{tikzpicture}[scale=0.80]
      \begin{semilogxaxis}[
          xlabel={Number of DOFs in $\XSgrad{\Th}$},
          ylabel={$\tilde{E}_h$},
          legend pos=south east
        ]
        \addplot [thick, mark=o, olive] table[x=DimXHess,y=EnApprox] {outputs/dirac_solution/hexa_k0/data_rates.dat};
        \addlegendentry{$k=0$}
        \addplot [thick, mark=square*, black] table[x=DimXHess,y=EnApprox] {outputs/dirac_solution/hexa_k1/data_rates.dat};
        \addlegendentry{$k=1$}
        \addplot [thick, mark=triangle*, blue] table[x=DimXHess,y=EnApprox] {outputs/dirac_solution/hexa_k2/data_rates.dat};
        \addlegendentry{$k=2$}
      \end{semilogxaxis}
    \end{tikzpicture}

    \vspace{0.15cm}
    \centerline{(b) Hexagonal meshes, various $k$.}
  \end{minipage}

  \caption{Error $\tilde{E}_h$ defined in \eqref{eq:def.Eh.dirac} for the Dirac right-hand side, with respect to the number of DOFs in $\XSgrad{\Th}$. Left: comparison of several mesh families for $k=0$. Right: comparison of several polynomial degrees on hexagonal meshes.}
  \label{fig:dirac.energy}
\end{figure}
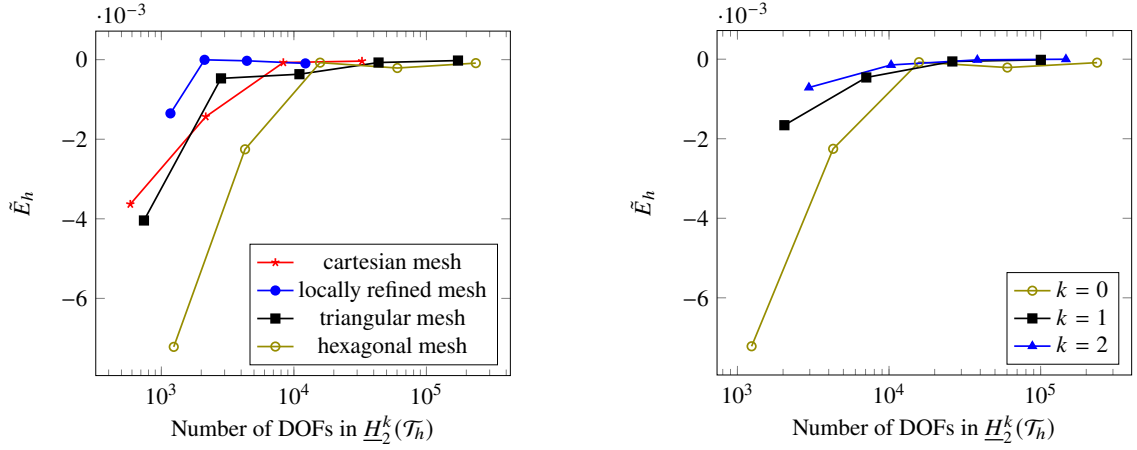

\section*{Acknowledgements}

The authors acknowledge the funding of the European Union via the ERC Synergy, NEMESIS, project number 101115663.

KH was supported in part by the ERC Starting project GeoFEM, project number 101164551, and a Royal Society University Research Fellowship (URF$\backslash$R1$\backslash$221398). Part of the work was carried out during KH's visit at University of Montpellier supported by the NEMESIS project. KH is grateful for the hospitality and the support from NEMESIS.

Views and opinions expressed are however those of the authors only and do not necessarily reflect those of the European Union or the European Research Council Executive Agency. Neither the European Union nor the granting authority can be held responsible for them.


\appendix

\section{Abstract framework for the transfer of Poincaré inequalities}\label{appendix:abstract.poincare}

In this appendix, we develop an abstract framework to transfer Poincar\'e inequalities between complexes connected by cochain maps.
Consider the diagram \eqref{eq:cohomology.double} below, where, for $i\in\{ 0,1 \}$, the spaces $X_i$ and $\hat X_i$ are endowed with inner products $(\cdot,\cdot)_{X,i}$ and $(\cdot,\cdot)_{\hat X,i}$, inducing, respectively, the norms $\norm{X,i}{{\cdot}}$ and $\norm{\hat X,i}{{\cdot}}$, and where the maps $E_i$ and $\hat{R}_i$ are continuous cochain maps.
\begin{equation}
  \label{eq:cohomology.double}
  \begin{tikzcd}[column sep=4.5em, row sep=3em]
    X_0\arrow{r}{d} \arrow[d, bend left, "{\hat R_0}"] & X_1\arrow[d, bend left, "\hat R_1"]\\
    {\hat{X}_0} \arrow{r}{\hat d} \arrow[u, bend left, "E_0"] & {\hat{X}_1}\arrow[u, bend left, "E_1"]
  \end{tikzcd}
\end{equation}
A blueprint to transfer Poincar\'e inequalities (and other algebraic and analytical properties) from one complex to another
has been developed in \cite{Di-Pietro.Droniou:23*1}. Typically, as illustrated in Diagram \ref{eq:doublecomplex_ddr0_stokes}, the mappings
$(\hat R_i)_i$ are reductions that remove some information from the richer complex $(X_i)_i$, while $(E_i)_i$ are extensions from a poorer complex $(\hat X_i)_i$. The framework of \cite{Di-Pietro.Droniou:23*1} only allows to transfer properties from the richer to the poorer complex, as Assumption (C1) in this reference requires that, when going from the poorer complex back into itself through the richer complex, no information must be lost.

However, for the purpose of Section \ref{sec:poincare}, we need to transfer Poincaré inequalities from a poorer complex (namely, the DDR complex of degree $0$) into a richer complex (the DS of degree $k$). We therefore have to develop a more general framework, which requires an additional assumption (Assumption \ref{ass:pc.reduced}) making up for the loss of information incurred when going from $(X_i)_i$ back to itself through $(\hat X_i)_i$. In practical cases, this assumption essentially boils down to assuming that \emph{local} Poincaré inequalities hold for the top sequence, as illustrated in the proofs of Propositions \ref{prop:pc.grad.coho} and \ref{prop:pc.curl.coho}.

\begin{assumption}[Poincaré inequality on $\Image(E_0\hat R_0-\Id)$]
  \label{ass:pc.reduced}
  There exists $C_{\rm P}\ge 0$, such that, for all $x\in X_0$, there exists $z\in X_0$ satisfying
  \begin{equation}
    \label{eq:loc.pc.cohomology}
    d(E_0 \hat R_0 x - x)=dz \quad\text{ and }\quad\norm{X,0}{z} \leq C_{\rm P}\norm{X,1}{dz}.
  \end{equation}
\end{assumption}

\begin{proposition}[Transfer of Poincaré inequality]
  \label{prop:pc.from.reduced.spaces}
  We suppose that Assumption \ref{ass:pc.reduced} holds, and that the bottom sequence of \eqref{eq:cohomology.double} satisfies the following Poincaré inequality:
  There exists $\hat{C}_{\rm P} \geq 0$ such that
  \begin{equation}
    \label{eq:pc.on.hat.x.cohomology}
    \norm{\hat X,0}{\hat x} \leq \hat C_{\rm P}\norm{\hat X,1}{\hat d\hat x} \qquad \forall \hat x \in (\Ker \hat d)^{\perp}.
  \end{equation}
  Then, the top sequence satisfies the following Poincaré inequality:
  \begin{equation}\label{eq:poincare.by.transfer}
    \norm{X,0}{x} \leq
    \left[
      \hat C_{\rm P}\norm{}{E_0}\norm{}{\hat R_1}+C_{\rm P}(\norm{}{E_1}\norm{}{\hat R_1}+1)
      \right]
    \norm{X,1}{dx}\qquad \forall x \in (\Ker d)^{\perp},
  \end{equation}
  where, for $\mathcal L\in\{E_1,\hat R_1,E_0\}$, $\norm{}{\mathcal L}$ denotes the mapping norm induced by the norms on $(X_i)_{i=0,1}$ and $(\hat X_i)_{i=0,1}$.
\end{proposition}

\begin{proof}
  The proof is inspired by some arguments in the proof of \cite[Proposition 4]{Di-Pietro.Droniou:23*1}. Let $x\in(\Ker d)^{\perp}$ and take $\hat x \in(\Ker \hat d)^{\perp}$ such that
  \begin{equation}\label{eq:dR0x.dy}
    \hat d \hat R_0 x = \hat d \hat x,
  \end{equation}
  which is possible because $\hat d:(\Ker\hat d)^{\perp}\rightarrow \Image \hat d$ is an isomorphism. Apply $E_1$ to both sides of this equality and use the cochain map property to obtain
  \begin{equation}\label{eq:dERx.hatx}
    dE_0\hat R_0 x = d E_0 \hat x.
  \end{equation}
  By Assumption \ref{ass:pc.reduced}, there exists $z\in X_0$ such that \eqref{eq:loc.pc.cohomology} holds. By \eqref{eq:dERx.hatx}, we then get $(x+z-E_0 \hat x)\in\Ker d$, and thus $(x+z-E_0  \hat x,x)_{X,0}=0$ since $x\in(\Ker d)^\perp$.
  Developing, we infer that $\norm{X,0}{x}^2 = (E_0\hat x -z,x)_{X,0} \leq \left(\norm{X,0}{E_0 \hat x} + \norm{X,0}{z}\right) \norm{X,0}{x}$, and thus
  \begin{equation}
    \label{eq:pc.coho.pf.1}
    \norm{X,0}{x} \leq \norm{X,0}{E_0 \hat x}+ \norm{X,0}{z}.
  \end{equation}
  To bound the first term in the right-hand side of \eqref{eq:pc.coho.pf.1}, we write
  \begin{align}
    \norm{X,0}{E_0 \hat x}
    \leq  \norm{}{E_0} \norm{\hat X,0}{\hat x}
    \overset{\eqref{eq:pc.on.hat.x.cohomology}}&\leq \hat C_{\rm P} \norm{}{E_0}\norm{\hat X,1}{ \hat d \hat x}
    \nonumber\\
    \overset{\eqref{eq:dR0x.dy}}&= \hat C_{\rm P} \norm{}{E_0}\norm{\hat X,1}{\hat d \hat R_0 x}\nonumber\\
    &= \hat C_{\rm P} \norm{}{E_0}\norm{\hat X,1}{\hat R_1 d  x}
    \leq \hat C_{\rm P} \norm{}{E_0}\norm{}{\hat R_1}\norm{X,1}{d x},
    \label{eq:pc.coho.pf.2}
  \end{align}
  where we have used the cochain map property to write $\hat d \hat R_0 x = \hat R_1 d  x$. To bound the second term in the right-hand side of \eqref{eq:pc.coho.pf.1}, we notice that
  \begin{align}
    \label{eq:pc.coho.pf.3}
    \norm{X,0}{z}\overset{\eqref{eq:loc.pc.cohomology}}\leq C_{\rm P} \norm{X,1}{dz} \overset{E_1\hat R_1 d x - dx = dz}{\leq} C_{\rm P}\left(\norm{}{E_1}\norm{}{\hat R_1}+1\right)  \norm{X,1}{dx},
  \end{align}
  where the equality justifying the conclusion comes from the cochain map property applied to $dz=d(E_0\hat R_0 x-x)$.
  Plugging \eqref{eq:pc.coho.pf.2} and \eqref{eq:pc.coho.pf.3} into \eqref{eq:pc.coho.pf.1} concludes the proof.
\end{proof}

\section{Reconstruction of a higher degree potential on $\XSgrad{T}$}
\label{appendix:P.nabla.k+3.T}

An alternative way to \eqref{eq:def.SPGRAD} for defining a potential reconstruction on $\XSgrad{T}$ is through a higher-order discrete gradient built from $\Pgrad{k+2}{T}:\XSrot{T}\rightarrow\Poly{k+2}(T)^2$ (see \cite[Section 2 and 4.2.1]{Di-Pietro.Droniou:23*1}). Define, for each $T\in\Th$, the full gradient reconstruction $\nablaTfull:\XSgrad{T}\to\Poly{k+2}(T)^2$ by
\[
\nablaTfull \dof{q}{T} \coloneq \Pgrad{k+2}{T} \SGRAD{T} \dof{q}{T}\quad\forall \dof{q}{T}\in\XSgrad{T}.
\]
According to the commutation property \eqref{eq:commutation.grad} of $\SGRAD{T}$ and the consistency property \eqref{eq:const.pol.P.grad.k+2} of $\Pgrad{k+2}{T}$, this gradient is polynomially consistent of degree $k+2$, in the sense that
\begin{equation}\label{eq:nablaTfull.consistent}
  \nablaTfull\ISgradT q=\GRAD q\qquad\forall q\in\Poly{k+3}(T).
\end{equation}

A potential reconstruction $\SPgradd:\XSgrad{T}\to \Poly{k+3}(T)$ can then be constructed by defining, for $\dof{q}{T}\in\XSgrad{T}$, the polynomial $\SPgradd\dof{q}{T}\in\Poly{k+3}(T)$ such that
\begin{equation}\label{eq:def.SPgradd}
  \int_T \SPgradd \dof{q}{T}\DIV \bvec w = -\int_T \nablaTfull\dof{q}{T}\cdot \bvec w+ \sum_{E\in\ET}\omega_{TE}\int_{E}(\gammaSfull{T}\underline{q}_T)_{|E}\,\bvec w\cdot\normal_{E} \quad \forall \bvec w\in \cRoly{k+4}(T),
\end{equation}
where the reconstruction $\gammaSfull{\partial T}:\XSgrad{T}\to\Poly{k+3}_{\rm c}(\ET)$ is defined by imposing \eqref{eq:def.gammaS} (as for $\gammaS\dof{q}{T}$) together with $\partial_{\tangent_E} (\gammaSfull{\partial T}\underline{q}_T)_{|E}(\bvec{x}_V)=\Gqv\cdot\tangent_E$ for all $E\in\ET$ and $V\in\VE$.
From \eqref{eq:nablaTfull.consistent} and the fact that $\gammaSfull{\partial T}\ISgradT q=q_{|\partial T}$ whenever $q\in\Poly{k+3}(T)$, we get the polynomial consistency property
\begin{equation*}
  \SPgradd \ISgradT q = q \qquad \forall q\in\Poly{k+3}(T).
\end{equation*}

Comparing with \eqref{eq:poly.consistency.Pgrad}, the potential $\SPgradd$ has a higher degree of primal consistency than $\SPgrad$ defined in \eqref{eq:def.SPGRAD}, but it seems to lack this higher accuracy for adjoint consistency. Let us briefly explain why. As seen in the proof of \cite[Theorem 9]{Di-Pietro.Droniou:23} (see also Remark \ref{rem:extention.SPGRAD}), the adjoint consistency relies on being able to use, in the definition \eqref{eq:def.SPgradd} of $\SPgradd$, test functions in $\Roly{k+2}(T)$. If that were possible, then for all $\bvec w \in \Roly{k+2}(T)$, choosing $\bvec{\zeta}\in\cRoly{k+3}(T)^2$ such that $\bvec{w}=\VDIV\bvec{\zeta}\in \Roly{k+2}(T)$ in \eqref{eq:def.SPgradd}  would lead to (using the definition \cite[Eq.~(4.1)]{Di-Pietro.Droniou:23} of $\Pgrad{k+2}{T}$):
\begin{align*}
  \int_T \underbrace{\nablaTfull\dofT{q}}_{=\Pgrad{k+2}{T} \SGRAD{T} \dof{q}{T}} \cdot \VDIV \bvec{\zeta} &{}= -\int_T \boldsymbol{\mathsf{G}}_T^{k+1}\SGRAD{T}\dof{q}{T} :\bvec{\zeta} +\sum_{E\in\ET}\omega_{TE}\int_{E}\gammadR{T}(\SGRAD{T}\dofT{q}) \cdot (\bvec{\zeta} \normal_E),
\end{align*}
where $\boldsymbol{\mathsf{G}}_T^{k+1}$ is a serendipity gradient. For $\SPgrad$, the term equivalent to $\boldsymbol{\mathsf{G}}_T^{k+1}\SGRAD{T}$ involves a discrete rotor instead of $\boldsymbol{\mathsf{G}}_T^{k+1}$ and therefore vanishes by complex property (see the cancellation in \eqref{eq:Prot.valid.w}). However, this cancellation does not occur here, which prevents us from completing the proof of this adjoint consistency.

\section{Notation}\label{appendix:notations}

  \begin{center}
    \renewcommand{\arraystretch}{1.2}
    \begin{longtable}{ccc}
      \toprule
      \textbf{Symbol} & \textbf{Continuous counterpart} & \textbf{Definition} \\
      \midrule
      \multicolumn{3}{c}{Discrete spaces} \\
      \midrule
      $\XSgrad{\Th}$
      & $H_2(\Omega)$
      & \eqref{eq:def.XSgrad} \\
      $\XSrot{\Th}$
      & $H_1(\Omega)^2$
      & \eqref{eq:def.Xgrad} \\
      $\tXdRrot{\Th}$
      & $\boldsymbol{H}_{\VROT}(\Omega)^2$
      & \eqref{eq:def.Xrot} \\
      \midrule
      \multicolumn{3}{c}{Interpolators} \\
      \midrule
      $\ISgrad$
      & ---
      & \eqref{eq:def.ISgrad} \\
      $\IdRgrad$
      & ---
      & \eqref{eq:def.IdRgrad} \\
      $\IdRrot$
      & ---
      & \eqref{eq:def.IdRrot} \\
      \midrule
      \multicolumn{3}{c}{Discrete differential operators} \\
      \midrule
      $\SGRAD{h}$
      & $\GRAD : H_2(\Omega) \to H_1(\Omega)$
      & \eqref{eq:def.nablah} \\
      $\SROT{h}$
      & $\ROT$
      & \eqref{eq:def.SROT} \\
      $\tGRAD{h}$
      & $\GRAD : H_1(\Omega)^2 \to \boldsymbol{H}_{\VROT}(\Omega)^2$
      & \cite{Di-Pietro.Droniou.ea:26} \\
      $\tROT{h}$
      & $\VROT$
      & \cite{Di-Pietro.Droniou.ea:26} \\
      \midrule
      \multicolumn{3}{c}{Discrete potentials} \\
      \midrule
      $\SPgradh$
      & $q\in H_2(\Omega) $
      &  \eqref{eq:def.SPGRAD} \\
      $\SProth$
      & $\bvec v \in H_1(\Omega)^2 $
      & \eqref{eq:def.potSrot} \\
      $\Pgrad{k+2}{h}$
      & $\bvec v \in H_1(\Omega)^2 $
      & \cite[Sec.~4]{Di-Pietro.Droniou:23} \\
      $ \Prot{k+1}{h}$
      &  $\bvec{\xi} \in \boldsymbol{H}_{\VROT}(\Omega)^2$
      & \cite[Sec.~4]{Di-Pietro.Droniou:23} \\
      \bottomrule
    \end{longtable}
\end{center}

\printbibliography


\end{document}